\documentclass{amsart}

\usepackage{amsmath,amsthm,amsfonts,amssymb}
\usepackage[all]{xy}
\usepackage{dsfont}
\usepackage{mathtools}
\usepackage[bookmarks=true, bookmarksopen=true,%
bookmarksdepth=3,bookmarksopenlevel=2, colorlinks=true, linkcolor=blue,%
citecolor=blue, filecolor=blue, menucolor=blue, urlcolor=blue]{hyperref}

\usepackage{tikz}
\usetikzlibrary{matrix}
\usepackage{todonotes}
\usepackage{fancyhdr}
\usepackage{verbatim}

\newtheorem{theorem}{Theorem}[section]
\newtheorem{lemma}[theorem]{Lemma}
\newtheorem{proposition}[theorem]{Proposition}
\newtheorem{corollary}[theorem]{Corollary}
\theoremstyle{definition}
\newtheorem{definition}[theorem]{Definition}
\newtheorem{remark}[theorem]{Remark}
\newtheorem{example}[theorem]{Example}
\DeclareMathOperator{\colim}{colim}
\newcommand{\isom}{\cong}

\newcommand{\Bcom}{B_{\mathrm{com}}}
\newcommand{\Ecom}{E_{\mathrm{com}}}
\newcommand{\RP}{\mathbb{RP}}

\newcommand{\Z}{\mathbb{Z}}
\newcommand{\F}{\mathbb{F}}

\newcommand{\U}{\mathcal{U}}
\newcommand{\V}{\mathcal{V}}
\newcommand{\C}{\mathcal{C}}

\newcommand{\D}{\mathcal{D}}
\newcommand{\W}{\mathcal{W}}
\newcommand{\hofib}{\mathop{\mathrm{hofib}}}

\newcommand{\Sq}{\mathrm{Sq}}

\newcommand{\Hom}{\mathrm{Hom}}
\newcommand{\one}{\mathds{1}}
\newcommand{\injects}{\hookrightarrow}
\newcommand{\KOcom}{KO_{\mathrm{com}}}
\newcommand{\rKOcom}{\widetilde{KO}_{\mathrm{com}}}
\def\co{\colon\thinspace}
\newcommand{\e}{\emph}
\newcommand{\srt}[1]{\stackrel{#1}{\to}}
\newcommand{\srm}[1]{\stackrel{#1}{\maps}}
\newcommand{\maps}{\longrightarrow}
\newcommand{\heq}{\simeq}
\newcommand{\cross}{\times}
\newcommand{\bbZ}{\mathbb{Z}}
\newcommand{\xmaps}{\xrightarrow}
\newcommand{\rKO}{\widetilde{KO}}
\newcommand{\st}{\mathrm{st}}

\begin{document}

\begin{abstract}
Commutative $K$--theory, a cohomology theory built from spaces of commuting matrices, has been explored in recent work of Adem, G\'{o}mez, Gritschacher, Lind, and Tillman.
In this article, we use unstable methods to construct explicit representatives for the real commutative $K$--theory classes on surfaces. These classes arise from commutative $O(2)$--valued cocycles, and are analyzed via the point-wise inversion operation on commutative cocycles.
\end{abstract}

\title[Commutative cocycles on surfaces]{Commutative cocycles and stable bundles over surfaces}
\author{Daniel A. Ramras}
\address{Mathematical Sciences, Indiana University - Purdue University Indianapolis\\
IN 46202}
\email{dramras@iupui.edu}
\thanks{D. Ramras was partially supported by the Simons Foundation (Collaboration Grants \#279007 and \#579789).}

\author{Bernardo Villarreal}
\address{Instituto de Matem\'{a}ticas, Universidad Nacional Aut\'{o}noma de M\'{e}xico\\
CDMX 04510}
\email{villarreal@matem.unam.mx}
\date{}
\subjclass[2010]{Primary  55N15; Secondary  55R35}
\keywords{Commutative $K$--theory, vector bundle, commuting matrices}
\maketitle

\tableofcontents

\section{Introduction}

Let $X$ be a finite CW complex.
In Adem--G\'{o}mez~\cite{Ad1}, commuting orthogonal matrices were used to construct a variant of real topological $K$--theory, $\KOcom (X)$, whose classes are represented by vector bundles over $X$ equipped with a \emph{commutative trivialization}.  A trivialization of a bundle $E\to X$ is called \emph{commutative} if all pairs of transition functions commute (as elements of the structure group) wherever they are simultaneously defined (see Definition~\ref{def: cscov}). Similar considerations with unitary matrices yield complex commutative $K$--theory. There is a natural forgetful map $\KOcom (X)\to KO^0 (X)$, and in Adem--G\'{o}mez--Lind--Tillman~\cite{Ad3} it was shown that this map
admits a splitting (additively), and the same holds in the complex case. 

Using methods from stable homotopy theory, Gritschacher~\cite{Grthesis} showed that the reduced group $\rKOcom (S^2)$ is isomorphic to $\Z/2\oplus \widetilde{KO} (S^2)$.
One of our main results, Theorem~\ref{thm: main}, is that this holds with $S^2$ replaced by any closed, connected surface $\Sigma$, and in fact our unstable methods provide a new proof of Gritschacher's result.
Moreover, in Theorem \ref{thm:main2} we establish an isomorphism of non-unital rings 
\begin{equation}\label{eqn:ring}\rKOcom(\Sigma)\cong\rKO(\Sigma)\times \langle y\rangle\end{equation} 
with $y^2=2y=0$, and in the Appendix we give a presentation for the real $K$--theory ring of $\Sigma$.

We use  \emph{commutative cocycles} with values in $O(2)$ to give an explicit construction of the generator 
$$y\in \ker (\KOcom (\Sigma)\maps KO  (\Sigma)).$$ 
The crucial feature of a commutative cocycle  is that its point-wise inverse is again a well-defined, commutative cocycle.
To construct $y$, we exhibit an explicit open cover of $S^2$ along with a  commutative cocycle $\alpha$ whose associated vector bundle is \emph{trivial}, but whose point-wise inverse produces a \emph{non}-trivial bundle (Proposition~\ref{prop:cco2b}). 
The cocycle $\alpha$ is classified by a map $S^2\to \Bcom O(2)$, where $\Bcom O(2) \subset BO(2)$ is the \e{classifying space for commutativity} in $O(2)$ introduced by Adem--Cohen--Torres-Giese~\cite{Ad5} (see Section~\ref{sec:TC}).
By analyzing the inclusions 
$$O(2)\injects O(3)\injects \cdots \injects O$$ 
and the associated maps on $\Bcom O(n)$ (Proposition~\ref{prop:pi2BcomOn}),
we show that $\alpha$ 
corresponds to
a non-trivial class  $y\in \rKOcom (S^2)$; since the vector bundle associated to $\alpha$ is trivial, the image of $y$ in $\rKO  (S^2)$ is trivial.

Recent work of Antol\'in-Camarena, Gritschacher, and Villarreal~\cite{AnCGrVi} introduces new characteristic classes  for bundles equipped with commutative trivializations, and we use these classes to extend our result from $S^2$ to general surfaces.
In the last section, we apply our methods to deduce new information about these characteristic classes, which are then used to deduce the ring isomorphism (\ref{eqn:ring}).

\begin{remark} In the complex case, the natural map $\widetilde{KU}_{\mathrm{com}} (X)\to \widetilde{KU}  (X)$ is an isomorphism for every 2-dimensional CW complex $X$. This follows from Adem--G\'{o}mez~\cite[Proposition 3.3]{Ad1}.
\end{remark}
 
\subsection*{Acknowledgments} We thank Alejandro Adem for encouraging us to examine the question of stability for  the classes in Section~\ref{sec: stability}, and we thank Simon Gritschacher and Omar Antol\'in for helpful conversations. Additionally, we thank the referee for many detailed comments that improved the exposition.

\section{Transitionally commutative structures and power operations}\label{sec:TC}

Let $G$ be a Lie group. For any integer $k>0$, let $C_k(G)\subset G^k$ be the subspace of ordered commuting $k$-tuples in $G$. Defining $C_0(G) := \{e\}$, the family of spaces $\{C_k(G)\}_{k\geq0}$ has a simplicial structure arising by restriction of the simplicial structure of $NG$ via the inclusions $\iota_k\colon C_k(G)\hookrightarrow (NG)_k$, where $NG$ is the nerve of $G$ as a category with one object. The \emph{classifying space for commutativity} of $G$ is defined as the geometric realization
\[\Bcom G := | C_\bullet(G) |.\]
For a finite CW complex $X$, the \emph{reduced real commutative} $K$--theory of $X$ is defined as 
$$\rKOcom (X) = \colim_{n\to\infty} [X, \Bcom O(n)].$$
The reduced complex commutative $K$--theory of $X$ is defined similarly, replacing $O(n)$ by $U(n)$.

In general, the space $\Bcom G$ comes equipped with a natural inclusion 
$$\iota\co\Bcom G\to BG.$$ 
The pullback of the universal $G$-bundle $EG\to BG$ along $\iota$ is denoted $\Ecom G$. Since $EG$ is contractible, we have that 
\[\Ecom G = \hofib(\Bcom G \to BG).\]
The kernel of $\iota_* \co \rKOcom(X)\to \widetilde{KO} (X)$ is naturally isomorphic to $[X, \Ecom O]$ (see Section~\ref{sec:add}), so that $\Ecom O$ represents 
the ``non-standard" part of real commutative $K$--theory (and similarly in the complex case).

\begin{definition}\label{def: cscov}
Let $E\to X$ be a principal $G$-bundle. A \emph{commutative trivialization}  
of $E$
is an open cover $\{U_i\}$ of $X$ together with transition functions 
$$\{\alpha_{ij}\co U_i \cap U_j\to G\}$$ 
such that
\begin{enumerate}
\item $E$ is isomorphic to the $G$-bundle induced by $\{\alpha_{ij}\}$;
\item For any $x\in U_i\cap U_j\cap U_k\cap U_l$, the commutator $[\alpha_{ij}(x),\alpha_{kl}(x)]$ is trivial.\
\end{enumerate}  
\end{definition}

Following \cite{Ad1} a principal $G$-bundle that admits a  
commutative trivialization is called \emph{transitionally commutative}. In \cite{Grthesis}, S. Gritschacher defines a \emph{transitionally commutative structure} (TC structure) on a bundle $E\to X$, as a homotopy class $[{f^\prime}]\in [X, \Bcom G]$ such that $\iota \circ f^\prime$ is a classifying map for $E$; in other words, a TC structure on $E$ is a lift, up to homotopy, of a  classifying map $f$ for $E$: 
\[
\xymatrix{
& \Bcom G\ar[d]^{\iota}\\
X\ar@/^/[ru]^{f^\prime}\ar[r]^f & BG.}
\]

\begin{definition}
For $A\leq G$ abelian, the inclusion $BA\to BG$ factors as 
$$BA= \Bcom A \to \Bcom G \xrightarrow{\iota} BG.$$ If a classifying map $f\colon X\to BG$ for a $G$-bundle $E$ factors up to homotopy through the inclusion $BA\to BG$, for some abelian subgroup $A\subset G$, then we say the lift $f^\prime\colon X\to BA\to\Bcom G$ is an \e{algebraic} TC structure on $E$. 
\end{definition}
 
\begin{example}
Three particular examples of algebraic TC structures will be important.
If $E$ is a trivial bundle, then it has a \e{standard} algebraic TC structure in which $A = \{e\}$ is the trivial subgroup of $G$. If $E$ is a real line bundle, then its classifying map factors uniquely (up to homotopy) through $BO(1)$, again yielding a \e{standard} algebraic TC structure on $E$. Finally, if $E$ is a 2-dimensional real vector bundle, then a choice of orientation on $E$ yields an algebraic TC structure with $A=SO(2)$, and again we refer to this as the standard TC structure associated to an oriented 2--dimensional bundle. In these situations, we will use the notation $E^{\st}$ to denote $E$ with its standard TC structure. 
\end{example}

The following theorem allows us to pass between the homotopical notion of a TC structure and the more geometric notion in Definition~\ref{def: cscov}. 

\begin{theorem}[\cite{Ad1}, Adem, G\'{o}mez]\label{thm:comclass}
Let $X$ be a finite CW-complex and $E\to X$ a principal $G$-bundle classified by $f\co X\to BG$. Then $f$ factors up to homotopy through the inclusion $\Bcom G\to BG$ if and only if $E$ is transitionally commutative.
\end{theorem}

\begin{definition}
Let $\C$ be a (not necessarily open) cover of a space $X$. A \emph{commutative cocycle} over $\C$ consists of a cocycle $\alpha=\{\alpha_{ij}\co C_i \cap C_j\to G\}$ where condition (2) of Definition \ref{def: cscov} above holds.
\end{definition}

In general, associated to a cover of a space $X$ by subspaces $\mathcal{C} = \{C_i\}_{i\in I}$, one has a category $\underline{\mathcal{C}}$ internal to topological spaces with object space $\coprod_i C_i$ and morphism space $\coprod_{(i,j)} C_i \cap C_j$ (here $x\in C_i\cap C_j$, considered as an element in the $(i,j)$--component of the disjoint union, is a morphism from $x\in C_i$ to $x\in C_j$). The nerve $N\mathcal{C}$ of this category is the \v{C}ech complex, as considered in~\cite{DuggerIsaksen}. A $G$--valued cocycle $\alpha$ on the cover $\mathcal{C}$ induces a functor $\underline{\mathcal{C}}\to \underline{G}$, where $\underline{G}$ is the category with one object associated to $G$. The induced map of nerves, $N\alpha\co N{\mathcal{C}}\to NG$, has image contained in $C_\bullet (G)$ when $\alpha$ is commutative, so the map between nerves factors as
\[N\C\xrightarrow{N\alpha}C_\bullet(G)\xrightarrow{\iota_\bullet} NG\]
By a map $\C\to \D$ between covers of $X$, we mean a function $f\co \C\to \D$ such that $C\subset f(C)$ for each $C\in \C$. Note that maps of covers induce functors between the associated categories, and simplicial maps between \v{C}ech complexes. We say a cover $\C$ of $X$ is \emph{good} if the natural projection $p_\bullet\colon N\C\to N\{X\}$ (induced by inclusions $C_i\hookrightarrow X$) is a homotopy equivalence after taking geometric realization. Note that there is a canonical homeomorphism $|N\{X\}|=X$; we will treat this as an identification from here on.  In~\cite{DuggerIsaksen} it is proven that all open covers are good. If $\alpha$ is a commutative cocycle over a good cover $\C$, then the TC structure associated to $\alpha$ is the homotopy class of $f_\alpha=|N\alpha|\circ p^{-1}\colon X\to \Bcom G$ 
where $p=|p_\bullet|$ and $p^{-1}$ is a homotopy inverse to $p$. We will say that two commutative cocycles  
 are equivalent if the induced TC structures are equal.

\begin{remark}\label{rem:partition}
A $G$--bundle $E\to X$ with a commutative trivialization over an open cover $\mathcal{U}$ gives rise to a commutative cocycle $\alpha$ over $\U$. If $X$ admits a partition of unity 
subordinate to $\U$, then one can construct a map $\lambda\co X\to |N\mathcal{U}|$ (this map also depends on a choice of ordering of the cover; see \cite[Section 2]{DuggerIsaksen} for details) which is a left inverse, and hence a homotopy inverse, to $p$. It follows that the TC structure on the bundle $E$ associated to $\alpha$ is  the homotopy class of the map $f_\alpha:=|N\alpha|\circ \lambda\co X\to \Bcom G$. 
We also note that, as a homotopy inverse to $p$, $\lambda$ is a homotopy equivalence and is independent, up to homotopy, of the choices involved. 
 \end{remark}

\paragraph{\bf Power operations} 
For each $n$ we have a simplicial map 
$$(\phi^n)_\bullet\co C_\bullet(G)\to C_\bullet(G),$$ 
given by $(\phi^n)_k(g_1,...,g_k) = (g_1^n,...,g_k^n)$ (for $n=0$, we set $g_i^0 = e$), yielding a map 
\[\phi^n\co\Bcom G\to \Bcom G\]
after geometric realization. 
The following properties can easily be checked.
\begin{enumerate}
\item $\phi^m\circ\phi^{n}=\phi^{mn}$
\item Any map $\gamma\co\Bcom G\to \Bcom H$ arising from a homomorphism of Lie groups $G\to H$ satisfies $\phi^n\circ\gamma = \gamma\circ\phi^n$.
\end{enumerate}

\begin{lemma}\label{lem:po&cc}
Let $\C$ be a good cover of a space $X$, and let $\alpha=\{\alpha_{ij}\}$ be a commutative cocycle over $\C$. For each integer $n$, set $\alpha^n:=\{(-)^n\circ\alpha_{ij}\}$. Then $\alpha^n$ is also a commutative cocycle over $\C$, with associated TC structure $[f_{\alpha^n}]=[\phi^n\circ f_{\alpha}]$. 
\end{lemma}
 \begin{proof}
The induced simplicial map $N\alpha\colon N\mathcal{C} \to NG$ has image contained in $C_\bullet(G)$, 
and the simplicial map $N\alpha^n$ factors as $N\mathcal{C}\xrightarrow{N  \alpha  }  C_\bullet(G)\xrightarrow{ \phi^n}C_\bullet(G) \subset NG.$
Therefore $|N\alpha^n| = \phi^n \circ |N\alpha|.$
\end{proof}

\section{Transitionally commutative structures over the sphere}

In this section we exhibit a commutative trivialization for the trivial $O(2)$--bundle over $S^2$, with the property that the induced map $S^2 \to \Bcom O(2)$ is \e{not} nullhomotopic. We prove this last statement using the power operations introduced in the previous section. In Section~\ref{sec: real} this TC structure will be used to produce non-trivial commutative $K$--theory classes on surfaces.

\subsection{Associated clutching function over a sphere}\label{cf-sec}
We use the cover $\C$ of $S^n\subset\mathbb{R}^{n+1}$ with $n>1$, consisting of three closed sets $C_1,C_2$ and $C_3$ defined as follows. Consider the ``left" and ``right" hemispheres 
\[C_1 = \{\vec{x}\in S^n\;|\;x_0\leq 0\}\,\,\, 
\textrm{ and } \,\,\, C = \{\vec{x}\in S^n\;|\;x_0\geq 0\}.\] 
Next, we cover $C$ with the ``north" and ``south" $n$-dimensional discs, 
\[C_2 = \{\vec{x}\in C\;|\;x_{n}\geq0\} \text{  and  } C_3 = \{\vec{x}\in C\;|\;x_n\leq0\}.\] 
Consider the retraction $r\co C_2\to C_2\cap C_3$ given by 
\[ r(x_0,...,x_n) = \left(\sqrt{1-\sum_{i=1}^{n-1} x_i^2}, x_1, \ldots, x_{n-1}, 0\right)\]
Let $\alpha=\{\alpha_{ij}\co C_i\cap C_j\to G \}$ be a $G$--valued cocycle on the (closed) cover $\C$. 
The usual gluing construction now gives us the $G$--space
\[E_\alpha = C_1\times G\sqcup C_2\times G\sqcup C_3\times G/(x,g)_j \sim (x,\alpha_{ij}(x)g)_i\]
where $(x,g)_i\in C_i\times G$, with $G$--action given by right multiplication in $G$. 

\begin{lemma} \label{lem:E_alpha}The natural projection $E_\alpha\to S^n$ makes this $G$--space into a (locally trivial) principal $G$--bundle.
\end{lemma}
\begin{proof} We apply the following general observation: If $X=A_1\cup A_2$ is the union of two closed sets $A_1$ and $A_2$, and there exist open neighborhoods $U_i \supset A_i$ that retract to $A_i$, then for any map $a\co A_1 \cap A_2 \to G$, there is a $G$--equivariant homeomorphism  
\[(A_1 \cross G \sqcup A_2 \cross G)/(x,g)_1 \sim (x, a(x)g)_2 \isom
 (U_1 \cross G \sqcup U_2 \cross G)/(x,g)_1 \sim (x, \tilde{a}(x)g)_2\]
 where $\tilde{a} (x) = a\circ \rho$ and $\rho\co U_1\cap U_2\to A_1\cap A_2$ is the retraction defined by gluing retractions together $\rho_i \co U_i\to A_i$ ($i=1, 2$) as follows. By hypothesis $U_1\cap U_2=(U_1\cap A_2)\cup(U_2\cap A_1)$, and we define $\rho$ to equal $\rho_1$ on the first set and $\rho_2$ on the second set. Since the intersection of these two sets is $A_1\cap A_2$, where $\rho_1=\rho_2=Id$, $\rho$ is well defined. 
 
 This observation applies to the cover of $C$ by the closed sets $C_1$ and $C_2$, thereby showing that the restriction of $E_\alpha$ to $C$ is a locally (and hence globally) trivial principal $G$--bundle. Applying this observation again to the cover of $S^n$ by the closed sets $C_1$ and $C$ completes the proof.
 \end{proof}

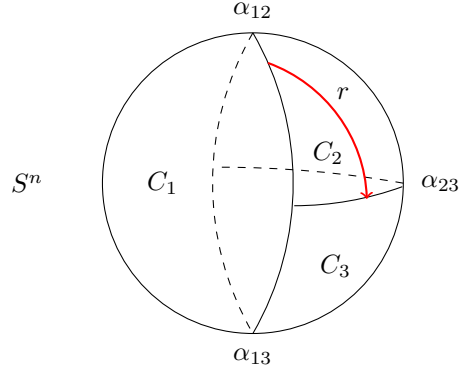
\[
\begin{tikzpicture}
\draw (0,0) circle (2cm) ;
\draw(0,2) arc (30 : -30 : 4);
\draw[dashed](0,2) arc (150 : 210 : 4);
\draw(.55,-.3) arc (270: 290 : 4.2);
\draw[dashed](2,0) arc (80 : 90 : 14);
\draw[red,thick,->](.2,1.6) arc (70 : 2 : 2);
\draw(1.2,1.2) node {$r$};
\draw(-3,0) node {$S^n$};
\draw(-1.2,0) node {$C_1$};
\draw(1,.4) node {$C_2$};
\draw(1.1,-1.1) node {$C_3$};
\draw(0,2.3) node {$\alpha_{12}$};
\draw(0,-2.3) node {$\alpha_{13}$};
\draw(2.5,0) node {$\alpha_{23}$};
\end{tikzpicture}
\]
We define the clutching function $\varphi\co C_1\cap(C_2\cup C_3) = S^{n-1}\to G$ as
\[\varphi(x)=\left\{
\begin{matrix}
\alpha_{12}(x)\alpha_{23}(r(x)) & \text{if} & x\in C_1\cap C_2\\
\alpha_{13}(x) & \text{if} & x\in C_1\cap C_3
\end{matrix}\right.\] 
which is well defined and continuous since $\alpha$ satisfies the cocycle condition. We will refer to $\varphi$ as the clutching function induced by $\alpha$. 
Note that $\varphi$ is a cocycle for the cover $\{C_1, C_2 \cup C_3\}$, so we have an associated bundle 
\[E_\varphi = \left(C_1\times G\sqcup (C_2\cup C_3)\times G\right)/(x,g)_{23} \sim (x,\varphi(x)g)_1.\]
We refer to $E_\varphi$ as the bundle clutched by $\varphi$.

\begin{lemma}\label{lem:tf&cf}
The principal $G$-bundle $E_\alpha\to S^n$ induced by $\alpha$ is isomorphic to the bundle clutched by $\varphi\co S^{n-1}\to G$. 
\end{lemma}
\begin{proof}
Denote an element in $(C_2\cup C_3)\times G$ by $(x,g)_{23}$, and let \[\psi\co C_1\times G\sqcup (C_2\cup C_3)\times G\to E_\alpha\]
be the map given by $\psi((x,g)_1) = [(x,g)_1]$ and 
\[\psi((x,g)_{23}) = \left\{\begin{matrix}
[(x,\alpha_{23}(r(x))g)_2] & \text{if} & x\in C_2\\
[(x,g)_3] & \text{if} & x\in C_3
\end{matrix}\right.\]
We claim that the map $\psi$ is well defined in the quotient 
$E_\varphi$. 

Indeed, let $x\in C_1\cap(C_2\cup C_3)$. Then $\psi((x,g)_1) = [(x,g)_1]$ and we have two cases:
\begin{center}
$\bullet$ if $x\in C_2$, $\psi((x,\varphi(x)^{-1}g)_{23}) = [(x,\alpha_{12}(x)^{-1}g)_2] = [(x,g)_1]$
\end{center}

\begin{center}
$\bullet$ if $x\in C_3$, $\psi((x,\varphi(x)^{-1}g)_{23}) = [(x,\alpha_{13}(x)^{-1}g)_3] = [(x,g)_1]$
\end{center}

Let $\psi$ also denote the induced map $E_\varphi\to E$. To prove our claim it only remains to show that $\psi$ is right $G$-equivariant, but this is true since the construction of $\psi$ only involves left multiplication by transition functions. 
\end{proof}

\subsection{TC structures on $O(2)$-bundles over $S^2$}

We study the example case of $\C$ as a cover of $S^2$. In this case the pairwise intersections of the three sets in $\C$ are all homeomorphic to a closed interval. Let $k$ be an integer and denote the matrix $\left(\begin{matrix}
\cos k\theta & - \sin k\theta\\
\sin k\theta & \cos k\theta
\end{matrix}\right)$ by $R_{k\theta}$, where $0\leq\theta\leq \pi$. Also let $A=\left(\begin{matrix}
1&0\\
0&-1
\end{matrix}\right)$. With these matrices we can define a cocycle over $\mathcal{C}$ as follows. 
Projection onto the second coordinate defines a homeomorphism from 
each intersection $C_i\cap C_j$ to the interval $[-1,1]$; we treat these homeomorphisms as identifications in the formulas to follow. 
Then for any $k\in \Z$ we define $_k\alpha=\{_k\alpha_{ij}\colon C_i\cap C_j\to O(2)\}$ as 
\begin{align}\label{tcs:S2}
_k\alpha_{12}(t) & = R_{k\theta(t)}\notag\\
_k\alpha_{23}(t) & = A\\
_k\alpha_{13}(t) & = R_{k\theta(t)}A\notag
\end{align}
where $\theta(t)=\arccos(t)$ maps the front side intersection point $(0,1,0)$ to $\theta=0$ and the back side intersection point $(0,-1,0)$ to $\theta = \pi$. Notice that $R_{0}=I$ and $R_{k\pi} = \pm I$, that is, at the intersection points $t\in C_1\cap C_2\cap C_3$, both $_k\alpha_{12}(t)$ and $_k\alpha_{13}(t)$ lie in the center of $O(2)$. Therefore $_k\alpha$ is a commutative cocycle.

We now explain how a commutative cocycle on $\C$ gives rise to a well-defined homotopy class of maps $S^2\to \Bcom G$. 

Let $f\colon \C_1\to \C_2$ be a map of covers (as in Section~\ref{sec:TC}). If $f$ is bijective and the inclusions $C_1 \cap \cdots \cap C_n \injects f(C_1)\cap \cdots \cap f(C_n)$ are all homotopy equivalences, then the induced map $N\C_1\to N\C_2$ is a level-wise homotopy equivalence, and hence a homotopy equivalence on realizations. (Note that since the identity morphisms form a disjoint component in the space of morphisms of $\underline{\C_i}$, the nerves $N\C_i$ are proper simplicial spaces.)

Now, our cover $\C$ of $S^2$ admits a map to the open cover $\U$ consisting of $\epsilon$--neighborhoods of the sets $C_i$ (in the Euclidean metric, say). For $\epsilon < \sqrt{2}$, each  intersection of these neighborhood deformation retracts to the corresponding intersection of closed sets, so $|N\C|\to |N\U|$ is a homotopy equivalence. As pointed out in Remark \ref{rem:partition}, given a partition of unity subordinate to $\U$, we have a homotopy equivalence $S^2\to |N\U|$. Then the desired map is the composite $S^2\to |N\U|\to |N\C|$, where the last map is a homotopy inverse to $|N\C|\to |N\U|$.

From the previous paragraph, we see that if $\V$ is an open cover admitting a map from $\C$, and the induced map $|N\C|\to |N\V|$ is a homotopy equivalence, then we obtain a homotopy equivalence $S^2 \to |N\V| \to |N\C|$ (where the first map is induced by a partition of unity, and the second is a homotopy inverse to $|N\C|\to |N\V|$). 
As noted in as in Remark~\ref{rem:partition}, the former map is independent, up to homotopy, of the choices involved.
We claim that the composite  is independent, up to homotopy, of the choice of open cover $\V$. Specifically, say $\W = \{W_1, W_2, W_3\}$ is another open cover such that $C_i \subset W_i$ and the induced map $N\C\to N\W$ is a homotopy equivalence. Then we claim that the maps $S^2 \to |N\C|$ associated to $\V$ and $\W$ are homotopic. To see this, consider the following diagram in the homotopy category, in which all three maps out of $S^2$ are induced by partitions of unity and the vertical maps are induced by inclusions of covers:
\[\xymatrix{& |N\V| \ar[d] & \\
S^2 \ar[ur]^\heq \ar[r]^-\heq \ar[dr]_\heq & |N\V\cup \W| 
								& |N\C|.  \ar[ul]_\heq \ar[dl]^\heq\\
& |N\W| \ar[u]
}
\]
The triangles on the left are (homotopy) commutative because a partition of unity subordinate to $\V$ is also subordinate to $ \V\cup \W$, and hence there is a choice making the top left triangle strictly commutative (and similarly for the bottom left triangle). It follows that the vertical maps are homotopy equivalences. The right half of the diagram is (homotopy) commutative because the two composite functors $N\C \to N\V\cup \W$ are related by a continuous natural transformation (whose component on an object $(x, C_i)$ is $(x, V_i \cap W_i$)). Since all the maps in the diagram are isomorphisms, a diagram chase shows that the outer square commutes after inverting the arrows on the right (as desired).

Now, given  a commutative cocycle $\alpha$ on $\C$, we obtain a (homotopy class of) map(s) $\alpha_* \co S^2\to  \Bcom O(2)\to BO(2)$, factoring through $|N\alpha| \co |N\C|\to \Bcom O(2)$. 

\begin{lemma}\label{lem:partition} The map $\alpha_*$ classifies the bundle $E_\alpha\to S^2$.
\end{lemma}
\begin{proof}
In general, given a $G$--bundle $E\to X$ with cocycle $\alpha$ over an open cover $\V$ with a partition of unity subordinate to $\V$, the map $|N\alpha|\circ \lambda$, where $\lambda$ is the map discussed in Remark~\ref{rem:partition}, classifies $E$ (for a detailed discussion see \cite{Dupont}). 
To complete the proof, it will suffice to show that there  is a (not necessarily commutative) cocycle on $\U$ (the cover consisting of $\epsilon$--neighborhoods of $C_i$) that restricts to $\alpha$. By definition, $E_\alpha$ has a canonical trivialization over each  $C_i$ (and the differences between these trivializations recover the cocycle $\alpha$). Since $U_i$ deformation retracts to $C_i$, we see that $E_\alpha$ is also trivial over each $U_i$. Choose a trivialization of $E_\alpha$ over each $U_i$. On $C_i$, these trivializations differ from the canonical trivializations by left-multiplication with some functions $C_i\to G$, and we can extend these functions to $U_i$ via the retractions $U_i\to C_i$. Left-multiplying our trivializations over $U_i$ by these functions yields new trivializations over $U_i$ that agree with the canonical trivialization over $C_i$. Hence the cocycle associated to these adjusted trivializations restricts to $\alpha$, as desired.
\end{proof}

\begin{lemma}\label{lem:to2b}
The map $S^2 \to \Bcom O(2)$ associated to $\{_k\alpha_{ij}\}$ becomes nullhomotopic after composing with $\iota\co \Bcom O(2) \to BO(2)$.
\end{lemma}

Before proving Lemma \ref{lem:to2b}, we set up some notation. For each $k,n\in \Z$, let \[\varphi_{k,n}\co S^1\to O(2)\] denote the   clutching function induced by $_k\alpha^n$, as in Section~\ref{cf-sec}. 
Note that by Lemmas \ref{lem:partition} and \ref{lem:tf&cf}, we simply need to show that the bundles clutched by $\varphi_{k,1}$ are trivial.

We can represent  $\varphi_{k,n}$ diagrammatically as 
\begin{align}
\begin{tikzpicture}
\draw{(-4.7,0)} node{$\theta=\pi$};
\draw{(-5,-1)} node{$\varphi_{k,1}$};
\draw(-3,0) circle(1cm);
\draw(-4,0) node {$-$};
\draw(-2,0) node {$-$};
\draw(-1.3,0) node {$\theta=0$};
\draw[blue,thick,->](-2,0) arc (0 : 40 : 1);
\draw[blue,thick,->](-2,0) arc (0 : -40 : 1);
\draw(-2.9,1.5) node {$R_{k\theta}A$};
\draw(-2.9,-1.5) node {$R_{k\theta}A$};
\draw(1.3,0) node {$\theta=\pi$};
\draw(1,-1) node {$\varphi_{k,n}$};
\draw(3,0) circle (1cm);
\draw(4,0) node {$-$};
\draw(2,0) node {$-$};
\draw(4.7,0) node {$\theta=0$};
\draw[blue,thick,->](4,0) arc (0 : 40 : 1);
\draw[blue,thick,->](4,0) arc (0 : -40 : 1);
\draw(3.2,1.5) node {$R_{k\theta}^nA^n$};
\draw(3.2,-1.5) node {$(R_{k\theta}A)^n$};
\end{tikzpicture}\label{pic:cf}
\end{align}
where the upper semicircle corresponds to $C_1\cap C_2$ and the lower semicircle to $C_1\cap C_3$. 

\begin{proof}[Proof of Lemma \ref{lem:to2b}]
Running around the circle on the left of Figure (\ref{pic:cf}) in counterclockwise orientation starting at $\theta=0$, we see that the loop determined by $\varphi_{k,1}$ is homotopic to the constant loop for all $k\in \Z$. 
The result now follows from
  Lemmas \ref{lem:partition} and \ref{lem:tf&cf}.
\end{proof}

\paragraph{Notation:} From now on we will simply denote $f_k:=f_{_k\alpha}\co S^2\to \Bcom O(2)$. By Lemma \ref{lem:po&cc}, we have that $[\phi^n\circ f_k] =[ f_{(_k\alpha)^n}]$.\\

\noindent For $0\leq\theta\leq\pi$, the definition of $R_{\theta}$ is made in such a way that the homotopy class $[R_{2\theta}]$ can be regarded as a generator in $\pi_1(SO(2))$ (where  and thus a generator in $\pi_1(O(2))$ via the inclusion $SO(2)\subset O(2)$). We fix a generator $1\in \pi_2(BO(2))=\Z$ by choosing it to be the pre-image of the aforementioned generator under the inverse of the connecting homomorphism $\pi_2(BO(2))\to\pi_1(O(2))$ arising from the long exact sequence of homotopy groups of the fibration $O(2)\to EO(2)\to BO(2)$.
 
\begin{proposition}\label{prop:cco2b}
The homotopy class of the classifying map
 \[S^2\xrightarrow{\phi^n\circ f_k} \Bcom O(2)\xrightarrow{\iota}BO(2)\] 
in $\pi_2(BO(2))$ is given by $nk/2$ if $n$ is even and $(n-1)k/2$ if $n$ is odd.
\end{proposition} 
\begin{proof}
By Lemma \ref{lem:tf&cf}, the bundle classified by $\iota\circ\phi^n \circ f_k$ is isomorphic to the bundle clutched by $\varphi_{k,n}$ as defined above. 
Notice that $AR_{k\theta}A=R_{-k\theta}$ and $A^2=I$, so that 
\[(R_{k\theta}A)^n=\left\{
\begin{matrix}
I&n\text{ even}\\
R_{k\theta}A&n\text{ odd, }n>0\\
AR_{-k\theta}&n\text{ odd, }n<0
\end{matrix}\right.\;\;\;\;\;\;\;\text{and}\;\;\;\;\;\;\;
R_{k\theta}^nA^n=\left\{
\begin{matrix}
R_{nk\theta}&n\text{ even}\\
R_{nk\theta}A&n\text{ odd} 
\end{matrix}\right.\]
Since $R_{k\theta}A=AR_{-k\theta}$, then for all $n$ odd $(R_{k\theta}A)^n=R_{k\theta}A$. The loops $\varphi_{k,n}$ in $O(2)$ now have the form 

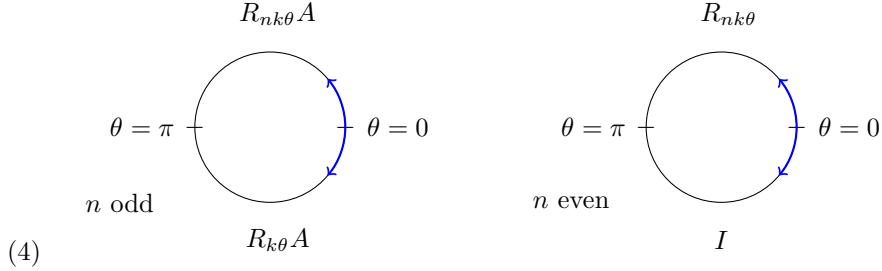
\begin{align}
\begin{tikzpicture}
\draw{(-5,-1)} node{$n$\text{ odd}};
\draw{(-4.7,0)} node{$\theta=\pi$};
\draw(-3,0) circle(1cm);
\draw(-4,0) node {$-$};
\draw(-2,0) node {$-$};
\draw(-1.3,0) node {$\theta=0$};
\draw[blue,thick,->](-2,0) arc (0 : 40 : 1);
\draw[blue,thick,->](-2,0) arc (0 : -40 : 1);
\draw(-2.9,1.5) node {$R_{nk\theta}A$};
\draw(-2.9,-1.5) node {$R_{k\theta}A$};
\draw(1,-1) node {$n$\text{ even}};
\draw(1.3,0) node {$\theta=\pi$};
\draw(3,0) circle (1cm);
\draw(4,0) node {$-$};
\draw(2,0) node {$-$};
\draw(4.7,0) node {$\theta=0$};
\draw[blue,thick,->](4,0) arc (0 : 40 : 1);
\draw[blue,thick,->](4,0) arc (0 : -40 : 1);
\draw(3.1,1.5) node {$R_{nk\theta}$};
\draw(3,-1.5) node {$I$};
\end{tikzpicture}\label{pic:cf2}
\end{align}
Let us define, for any $k$, $\gamma_{k}\colon[0,\pi]\to SO(2)$ as $\gamma_{k}(\theta)=R_{k\theta}$. When $n$ is even, $\varphi_{k,n}$ actually lies in $SO(2)$, and thus the homotopy class of the classifying map of the bundle clutched by $\varphi_{k,n}$ is given by $\deg(\gamma_{nk})=nk/2$. 

Now say $n$ is odd. Define  $\gamma_{(n-1)k} A\co [0,\pi]\to O(2)$ to be the point-wise product $(\gamma_{(n-1)k}A) (t) = \gamma_{(n-1)k} (t) \cdot A$.
Let $\xi\co [0, \pi]\to C_1 \cap (C_2\cup C_3)$ be the map
$\xi(\theta) = (0, \cos(2\theta), \sin (2\theta))$ (so that $\xi$ maps the interval $[0,\pi]$ once around the domain of $\varphi_{k,n}$, in counterclockwise orientation with respect to the above picture). Then $\gamma_{(n-1)k}\co [0,\pi]\to SO(2) \subset O(2)$ factors through $\xi$, and hence determines a loop in $O(2)$ with the same domain as $\varphi_{k,n}$; we will continue to denote the former loop by  $\gamma_{(n-1)k}$, and similarly for  $\gamma_{(n-1)k}A$.
From the picture above, we see that $\varphi_{k,n}$ is homotopic to $\gamma_{(n-1)k}A$,
and it follows that the $O(2)$--bundles $E_{\varphi_{k,n}}$ and $E_{\gamma_{(n-1)k}A}$  clutched by $\varphi_{k,n}$ and by 
$\gamma_{(n-1)k}A$ are isomorphic.
We claim that there is also an isomorphism  $E_{\gamma_{(n-1)k}A} \isom E_{\gamma_{(n-1)k}}$.
This will complete the proof, since the homotopy class of the classifying map of $E_{\gamma_{(n-1)k}}$ is given by $\deg(\gamma_{(n-1)k}) = \frac{(n-1)k}{2}$.

To prove the claim, let $X$ be a space decomposed as a union $X= C\cup D$, and let $G$ a group. Then for any function $\gamma\co C \cap D \to G$ and any $a\in G$, we claim that there is a $G$--equivariant homeomorphism $\overline{L}$ from the quotient space 
\[E_{\gamma a} := (C\cross G \sqcup D\cross G)/{(x,\gamma(x) a g)_C \sim (x,g)_D}\]
to the quotient space
\[E_\gamma := (C\cross G \sqcup D\cross G)/{(x,\gamma(x) g)_C \sim (x,g)_D}\] 
(where $x\in C\cap D$ and the subscripts indicate the factor of the disjoint union containing the element, and $G$ acts by right multiplication in the $G$ factors).
Indeed, the $G$--equivariant homeomorphism
\[L\co C\cross G \sqcup D\cross G \maps C\cross G \sqcup D\cross G\]
defined by
$L((x,g)_C) = (x,g)_C$  and $L((x,g)_{D}) = (x,ag)_{D}$
and its inverse $L^{-1}$ 
respect the equivalence relations and hence descend to give inverse homeomorphisms 
\[\xymatrix{\overline{L} \co E_{\gamma a} \ar@<0.5ex>[r] & E_{\gamma} \co \overline{L^{-1}}. \ar@<0.5ex>[l]}\]  
\end{proof}

Now we describe the behaviour of the power operations $\phi^n$ on the cohomology of $SO(2)$, which will be useful for our calculations.

\begin{lemma}\label{lem:powerophomi}
For every $n \in \Z$, 
the map $(\phi^{n})^*\colon H^2(BSO(2);\Z)\to H^2(BSO(2);\Z)$ is multiplication by $n$. 
\end{lemma}
\begin{proof}
By the Hurewicz Theorem and the Universal Coefficient Theorem, it suffices to prove the corresponding statement regarding the effect of $(\phi^{n})_*$ on $\pi_2(BSO(2))$. Writing $\mu\co SO(2)\cross SO(2)\to SO(2)$ for the addition map, $B\mu$ makes $BSO(2)$ a topological abelian group, and addition in $\pi_2(BSO(2))$ is induced by pointwise addition of maps $S^2\to BSO(2)$ (by the Eckmann--Hilton argument). It follows that $B(\phi^n)$ induces multiplication by $n$ in $BSO(2)$ and consequently in  $\pi_2(BSO(2))$.
\end{proof}

Any $SO(2)$-bundle has a standard algebraic TC structure given by the composition of its classifying map and the inclusion $j\co BSO(2)\to\Bcom O(2)$ induced by the inclusion $SO(2)\injects O(2)$. Above we have specified a generator of $\pi_2 (BO(2))$. The natural map $BSO(2) \to BO(2)$ is an isomorphism on $\pi_2$, so we obtain a corresponding choice of generator for $\pi_2 (BSO(2))$ 
and hence a bijection $\pi_2 (BSO(2)) \isom \Z$ (which we will treat as an identification).
For each $m\in \Z$, let $g_m\co S^2\to BSO(2)$ denote a representative of $m\in \pi_2 (BSO(2))$. Then $[j\circ g_m]$ is an algebraic TC structure on the $O(2)$--bundle $E_m$ classified by the composite $S^2\xmaps{g_m} BSO(2) \to BO(2)$.

Now we can compute the effects of the maps $f_k\colon S^2\to \Bcom O(2)$ in integral cohomology. This will allow us to construct non-trivial TC structures on bundles over other closed, connected surfaces.

\begin{lemma}\label{lem: jg=phif}
For every $k\in \Z$, $[j\circ g_k]= [\phi^2\circ f_k]$ in $\pi_2(\Bcom O(2))$.
\end{lemma}
\begin{proof}
Recall that $f_k=|N_k\alpha|\circ p^{-1}$, where $_k\alpha = \{_k\alpha_{ij}\}$ is the commutative cocycle in (\ref{tcs:S2}), and  $p^{-1}$ is a homotopy inverse to the projection $|N\C|\to X$ described above. The effect of post-composing with $\phi^2$ can be seen at the level of simplicial spaces. Since $_k\alpha_{12}^2(\theta)=(R_{k\theta})^2=R_{2k\theta}$, $_k\alpha_{23}^2(x)=A^2=I$ and $_k\alpha_{13}^2(\theta)=(R_{k\theta}A)^2=I$, we can conclude that $\phi^2_\bullet\circ N_k\alpha\co N\U\to C(O(2))$ has image contained in $NSO(2)$. 
So $\phi^2\circ f_k$ factors through $j\co BSO(2)\to \Bcom O(2)$. Writing this factorization as $ j \circ h_k $, it suffices to show that $[h_k]=[g_k]$ in $\pi_2(BSO(2))$. By Proposition \ref{prop:cco2b}, we have $ k= [\iota\circ\phi^2\circ f_k] = (\iota \circ j)_*[h_k]$. Since $\iota\circ j\colon BSO(2)\to BO(2)$ is the inclusion, $(\iota\circ j)_*$ is an isomorphism in $\pi_2$, and by definition of $g_k$, we have $(\iota\circ j)_*[g_k]=k$ as well. The result now follows.  
\end{proof}

In \cite{AnCGrVi}, the authors show the existence of a  class $r\in H^2(\Bcom O(2);\Z)$ satisfying
\begin{align}
j^*(r)=2e\label{eq:r&e}
\end{align} 
where $e\in H^2(BSO(2);\Z)$ is the Euler class of the universal oriented $SO(2)$-bundle.
The element $g_1^*(e) \in H^2 (S^2; \Z)$ is a generator, and yields an identification $H^2 (S^2; \Z)\isom \Z$.
Under this identification, we have 
\begin{align}e(E_m) := g_m^* (e) = m.\label{eq:g_m&e}
\end{align} 

\begin{lemma}\label{lemm:clrEul}
For each $k\in \Z$, we have
\[f_k^*(r)=k.\]
\end{lemma}

\begin{proof}
By Lemma~\ref{lem: jg=phif},
in integral cohomology $(\phi^2\circ f_k)^*$  has the same values as
\[H^2(\Bcom O(2);\Z)\xrightarrow{j^*}H^2(BSO(2);\Z)\xrightarrow{g_k^*}H^2(S^2;\Z)\]
which by (\ref{eq:r&e}) and (\ref{eq:g_m&e}) maps $r\mapsto 2k$. We claim that $(\phi^2)^*(r) = 2r$ in $H^2(\Bcom O(2);\Z)$. In \cite{AnCGrVi} it is shown that $H^2(\Bcom O(2);\Z)$ is generated by the classes $r$ and $W_1$, the latter being the image under $\iota^*\colon H^2(BO(2);\Z)\to H^2(\Bcom O(2);\Z)$ of the Bockstein of $w_1\in H^1(BO(2);\F_2)$. Let us write $(\phi^2)^*(r)=mr+\varepsilon_1W_1$, with $m\in \Z$ and $\varepsilon_1\in \{0,1\}$. Lemma \ref{lem:powerophomi} and (\ref{eq:r&e}) give $m=2$. Now consider the inclusion $k\colon BO(1)^2\to \Bcom O(2)$. In \cite{AnCGrVi}, they show as well that $k^*(r)=0$, and since $k^*(w_1)=u+v$, where $u,v$ are the degree 1 generators of the polynomial algebra $H^*(BO(1)^2;\F_2)\cong \F_2[u,v]$, naturality of the Bockstein $\beta$ yields $k^*(W_1)=\beta(u)+\beta(v)$. Now $k^*\circ(\phi^2)^*=(\phi^2)^*\circ k^* $ implies that $\varepsilon_1(\beta(u)+\beta(v))=0$ in $H^2(BO(1)^2;\Z)$. Therefore $\varepsilon_1=0$, which proves our claim. Then under $(\phi^2\circ f_k)^*$, $r\mapsto 2f^*_k(r)$, so we can conclude that $f_k^*(r)=k$. 
\end{proof}

\begin{remark}
Lemma \ref{lem:to2b} says that for any $k$, $\iota\circ f_k$ is nullhomotopic, but we can deduce form Lemma \ref{lemm:clrEul} that for $k\ne0$, $f_k$ is not.
\end{remark}

\begin{example}
  Let $\Sigma$ be an orientable, closed, connected surface and $c\colon\Sigma\to S^2$ the map that collapses the complement of an open disk in $\Sigma$ to a point. Since $\Sigma$ is orientable, $c$ is a degree 1 map. 
Lemma \ref{lemm:clrEul} shows that for every integer $k$,  we have 
$$c^* f_k^*(r)= c^*(e(E_k)) = k c^*(e(E_1)).$$
Moreover, $e(E_1) \in H^2 (S^2; \Z)$ is a generator (by definition), and hence so is $c^*(e(E_1))$ $\in H^2(\Sigma;\Z)$.
Thus the TC structures $f_k\circ c$ all lie in different homotopy classes. 

Similarly, if $\Sigma$ is a closed, connected, non-orientable surface, we can choose a map $c\colon\Sigma\to S^2$ as above, and $c$ will have degree 1 in cohomology with  coefficients in $\Z/2$, and we find that for $k$ odd, $c^*f_k^*(\bar{r})$ generates $H^2 (\Sigma; \F_2)$, where $\bar{r}$ is the image of $r$ in $\F_2$--cohomology.
In particular, we find that the TC structures 
$f_k \circ c$
are non-trivial when $k$ is odd.
\end{example}

In \cite{AnCGrVi} it is shown that $\pi_2(\Bcom O(2))\cong \Z^3$. 
In addition to the TC structures $f_k\colon S^2\to\Bcom O(2)$ on the trivial bundle, we  also have the standard algebraic TC structures $j\circ g_m$ considered in Lemma~\ref{lem: jg=phif}.
It is an interesting question whether combinations of these TC structures yield all homotopy classes in $\pi_2(\Bcom O(2))$. To be more precise about these combinations, we have 
\[S^2\xrightarrow{\text{pinch}}S^2\vee S^2\xrightarrow{f_k\vee g_m}{\Bcom O(2)\vee BSO(2)}\xrightarrow{Id,j}\Bcom O(2)\]
which is nothing but the sum in $\pi_2(\Bcom O(2))$. Motivated by this, we will denote the above TC structure by $f_k+(j\circ g_m)$. Here we prove that each pair of integers $(k,m)$ yields a distinct class $[f_k]+[(j\circ g_m)]\in \pi_2(\Bcom O(2))$.

\begin{proposition}\label{thm:csEk}
Let $E_m$ be an $O(2)$-bundle over $S^2$ with a choice of orientation given by $m\in\pi_2(BSO(2))$. Then for any $k\in \bbZ$, $f_k+(j\circ g_m)$ is a TC structure on $E_m$. Moreover, all of these TC structures on $E_m$ are different.
\end{proposition} 
\begin{proof}
Let $h_m\co S^2\to BO(2)$ be a classifying map of $E_m$. Then $\iota\circ (j\circ g_m)\simeq h_m$ and by Lemma \ref{lem:to2b}, $\iota\circ f_k$ is homotopic to the constant map so that indeed, $f_k+(j\circ g_m)$ is a lift of $h_m$ to $\Bcom O(2)$. To see that all these lifts are in different homotopy classes in $\pi_2(\Bcom O(2))$, consider the compositions $(\iota\circ\phi^{-1})\circ (f_k+(j\circ g_m))$.
Notice that the inclusion $j$ commutes with the map $\phi^{-1}$ since $j$ is induced from the homomorphism $SO(2)\hookrightarrow O(2)$. Therefore $\phi^{-1}\circ j\circ g_k = j\circ \phi^{-1}\circ g_k$. By Lemma \ref{lem:powerophomi}, $\phi^{-1}_*[g_m]=-[g_m]$. Then by Proposition \ref{prop:cco2b} we see that in $\pi_2(BO(2))$ we have
\[[(\iota\circ\phi^{-1})\circ (f_k+(j\circ g_m))]=-k-m,\]
which are distinct for different values of $k$.
\end{proof}

\section{Real commutative $K$-theory of surfaces: Additive structure}\label{sec: real}

In this section we prove one of our main results,  computing the group $\rKOcom(\Sigma)$ for closed connected surfaces. The key ingredient is to compute the homotopy groups $\pi_2(\Bcom O(n))$ for $n\geq 3$ together with the homomorphisms $\pi_2(\Bcom O(2))\to \pi_2(\Bcom O(n))$ induced by the canonical inclusions $O(2)\hookrightarrow O(n)$.   

\subsection{Stability} \label{sec: stability}

We study how the classes $f_k$ considered in the previous section behave as we stabilize using the inclusions $O(2)\injects O(n)$ given  by $X\mapsto 
\left(\begin{matrix}X&0\\
0&I_{n-2}
\end{matrix}
\right)$.  These maps induce maps $i_n\co BO(2)\to BO(n)$ and $j_n \co \Bcom O(2)\to \Bcom O(n)$.

\begin{proposition}\label{thm:scEcomO}  
If $k$ is odd and $n\geq 3$, the homotopy class of 
$$j_n \circ f_k\co S^2\to\Bcom O(n)$$
  is non-trivial in $\pi_2(\Bcom O(n))$.
\end{proposition}
\begin{proof}
Fix $n\geq 3$.  
We have a commutative diagram
\[
\xymatrix{\Bcom O(2)\ar[r]^{j_n}\ar[d]^{\iota_2}&\Bcom O(n)\ar[d]^{\iota_n}\\
BO(2)\ar[r]^{i_n}&BO(n)
}
\] 
in which $\iota_2$ and $\iota_n$ are the natural inclusions, and $j_n$ is the map on $\Bcom(-)$ induced by $i_n$.
By Lemma \ref{lem:to2b}, $[\iota_2 \circ f_k]$ is a trivial class in $\pi_2(BO(2))$, so that $(i_n)_*[\iota_2 \circ f_k]$ is also trivial in $\pi_2(BO(n))$, which is the same class as $(\iota_n)_*[ j_n \circ f_k]$. 

To study the homotopy class of $j_n \circ f_k\co S^2\to\Bcom O(n)$ in $\pi_2(\Bcom O(n))$, let $\eta$ be the composition
\[S^2\xrightarrow{j_n\circ f_k} \Bcom O(n)\xrightarrow{\phi^{-1}}\Bcom O(n)\xrightarrow{\iota_n}BO(n).\]
We claim that the homotopy class of $\eta$ is the nontrivial element in $\pi_2(BO(n))=\Z/2$ if and only if $k$ is odd. Notice that this would finish the proof, since our claim forces $[j_n\circ f_k]$ to be non-trivial as long as $k$ is odd. 

The map  $j_n$ is induced by a homomorphism, so by naturality it commutes with $\phi^{-1}$. Now we see that 
\[\eta=\iota_n \circ \phi^{-1} \circ j_n\circ f_k =
\iota_n \circ j_n \circ  \phi^{-1} \circ f_k = 
i_n\circ \iota_2\circ  \phi^{-1}\circ f_k.\]
 Proposition \ref{prop:cco2b} implies that the homotopy class of $\iota_2\circ \phi^{-1}\circ f_k$ in $\pi_2(BO(2))$ is represented by the integer $-k$. Thus to prove our claim we only need to recall how $i_n$ behaves on $\pi_2(-)$. The long exact sequence of homotopy groups of the fibrations $S^m\to BO(m)\to BO(m+1)$ imply that  $(i_3)_*\colon\pi_2(BO(2))\to \pi_2(BO(3))$ is the reduction mod 2 homomorphism, and for $m>2$, $\pi_2(BO(m))\to \pi_2(BO(m+1))$ is an isomorphism. 
It follows that $(i_n)_*\colon\pi_2( BO(2))\to \pi_2 (BO(n))$ is reduction mod 2 as well.
Hence $[\eta]=(i_n)_*[\iota_2 \circ \phi^{-1}\circ f_k]\in \pi_2(BO(n))$ is a generator if and only if $k$ is odd. 
\end{proof}

\subsection{Computing $\pi_2(\Bcom SO(n))$ and $\pi_2(\Bcom O(n))$}

Recall that in ${C}_\bullet(G)$, the simplicial model for $\Bcom G$, the face maps $d_i\co {C}_n(G)\to {C}_{n-1}(G)$ are given by 
\[d_i(g_0,...,g_n)=\left\{\begin{matrix}
(g_1,...,g_n)&i=0\\
(g_0,...,g_{i-1}g_i,...,g_n)&0<i<n\\
(g_0,...,g_{n-1})&i=n.
\end{matrix}\right.\]
These maps induce a simplicial abelian group structure on $H_q(C_\bullet(G))$ for any $q\geq0$.

\begin{lemma}\label{lem:H2Bcom}
Let $G$ be a connected Lie group. Then \[H_2(\Bcom G;\Z)=H_2H_0(C_\bullet(G))\oplus H_1(G;\Z).\]
\end{lemma}

\begin{proof}
Let us write $H_m(X)$ for singular homology with integral coefficients. We use the homology spectral sequence $E^r_{p,q}$ associated to the simplicial space $C_\bullet(G)$ that strongly converges to $H_{p+q}(|C_\bullet(G)|)=H_{p+q}(\Bcom G)$. Since $C_\bullet(G)$ is proper (see \cite[Appendix]{Ad1}), the $E^2$-page is given by $E_{p,q}^2=H_p(H_q(C_{\bullet}(G)))$ where for each $q$, $H_q(C_\bullet(G))$ is regarded as a chain complex. Let us write $A_n=H_q(C_n(G))$. Then $E_{p,q}^2=H_p((A_*,\partial))$ is the $p$-th homology group of the (un-normalized) chain complex $(A_*, \partial)$ with differentials given by $\partial=\sum_{i=0}^n(-1)^i(d_i)_*$. To compute $H_2(\Bcom G)$ we only need to worry about $E_{0,2}^2,E_{1,1}^2$ and $E_{2,0}^2$. We study each of these cases separately. \\

$\bullet$ $E_{0,2}^2=H_0H_2(C_\bullet(G))$. In this case $A_0=H_2(C_0(G))=H_2(pt)=0$. Therefore $E_{0,2}^2=H_0((A_*,\partial))=0$.\\

$\bullet$ $E_{1,1}^2=H_1H_1(C_\bullet(G))$. The chain complex is now $A_n=H_1(C_n(G))$. The differential $\partial\co A_1\to A_0$ is zero, since $d_0$ and $d_1$ are constant maps at simplicial level 1. We claim that $\partial\co A_2\to A_1$ is zero. To see this, we analyze the effect in $H_1$ of the respective face maps $d_i\co G\times G\to G$ for the nerve of $G$ at simplicial level 2. For $i=0,2$ the induced homomorphisms are just the projections on the first and second coordinate respectively. For $i=1$, $d_1$ is the product in $G$, so $(d_1)_*$ is the addition map in $H_1(G)$. Therefore the alternating sum of $(d_i)_*$'s is zero. Now then, the differential $\partial$ factors through
\[H_1(C_2(G))\to H_1(G\times G)\xrightarrow{0} H_1(G)\]
and now we see that $\partial$ must be the zero homomorphism. Therefore $E_{1,1}^2=H_1(G)$.\\

$\bullet$ The term $E_{2,0}^2=H_2H_0(C_\bullet(G))$ by definition. \\

To finish our proof we need to show that $E^2_{p,q}=E^\infty_{p,q}$ when $p+q=2$ and that there are no extension problems. The differential $d_2\co E_{2,0}^2\to E^2_{0,1} = 0$ must be zero and thus $E_{2,0}^2$ survives in the $E^\infty$-page. To see that $E^2_{1,1}=E^\infty_{1,1}$, we will show that the differential $d_2\co E_{3,0}^2\to E_{1,1}^2$ is zero. Consider the simplicial map arising by inclusions $\iota_\bullet\co C_\bullet(G)\to NG$. Let $\mathbb{E}_{p,q}^r$ denote the spectral sequence associated to the nerve $NG$. Then $\iota_\bullet$ induces a map of spectral sequences $\iota_*\co E_{p,q}^r\to\mathbb{E}_{p,q}^r$. Notice that $\iota_1\co C_1(G)=G\to G$ is the identity map which itself induces the identity isomorphism $E_{1,1}^2=H_1(H_1(C_\bullet(G)))\to \mathbb{E}_{1,1}^2=H_1(H_1(NG))$. Since $\iota_*$ commutes with $d_2$ it is enough to see that $d_2\colon \mathbb{E}_{3,0}^2 \to \mathbb{E}^2_{1,1}$ is zero. Since $G$ is connected, it follows that $(H_0(NG),\partial)$ is an exact chain complex (except in degree 0). Then $\mathbb{E}_{3,0}^2=H_3H_0(NG)=0$ and $d_2$ is zero.

Now we verify that there are no extension problems. We have that $H_2(BG)\cong \mathbb{E}_{1,1}^2$ (clearly $H_0(H_2(NG))=0$ and $H_2(H_0(NG))=0$ by similar arguments as before). Then $\iota_\bullet$ induces a morphism of extensions
\[\xymatrix{0\ar[r]&E_{1,1}^2\ar@{=}[d]\ar[r]&H_2(\Bcom G)\ar[d]^{\iota_*}\ar[r]&E_{2,0}^2\ar[r]&0\\
&\mathbb{E}^{2}_{1,1}\ar[r]^{\cong}&H_2(BG),}\]
and by commutativity of the diagram, $E_{1,1}^2\to H_2(\Bcom G)$ splits.
\end{proof}

\begin{lemma}\label{lem:H2Bcomso3}
$H_2(\Bcom SO(3);\Z)\cong\Z/2\oplus\Z/2$.
\end{lemma}
\begin{proof}
The group $H_1(SO(3);\Z)=\Z/2$ so that by Lemma \ref{lem:H2Bcom} we only need to show that the group $H_2H_0(C_\bullet(SO(3)))=\Z/2$. This case is more interesting, since $C_n(SO(3))$ is not path connected for $n\geq2$ (see for example \cite[Theorem 2.4]{TorSjerve}). We follow the notation of Lemma \ref{lem:H2Bcom}. We are interested in the sequence
\[A_3=H_0(C_3(SO(3)))\xrightarrow{\partial_2} A_2=H_0(C_2(SO(3)))\xrightarrow{\partial_1}A_1= H_0(SO(3)).\]
By \cite[Theorem 2.4]{TorSjerve}, $C_2(SO(3))$ has 2 connected components and $C_3(SO(3))$ has 8.  Let $\Hom(\Z^n,SO(3))_\one$ denote the connected component of $C_n(SO(3))$ that contains the trivial representation $\one\co \Z^n\to SO(3)$ (which is represented by the $n$-tuple $(I,...,I)$). The connected components of $C_n(SO(3))$ are either $\Hom(\Z^n,SO(3))_\one$ or homeomorphic to $SU(2)/Q_8$, where $Q_8$ is the quaternion group. The face maps restrict to the components containing $\one$, that is, $d_i\co \Hom(\Z^n,SO(3))_\one\to\Hom(\Z^{n-1},SO(3))_\one$. Since $\partial_1$ restricted to the homology of each component is the alternating sum of three identity homomorphisms, the differential $\partial_1$ sends both generators $(1,0)$ and $(0,1)$ in 
\[A_2=H_0(\Hom(\Z^2,SO(3))_\one)\oplus H_0(SU(2)/Q_8)\cong\Z\oplus\Z\] 
to the generator in $A_1=\Z$. Thus $\ker\partial_1=\langle(-1,1)\rangle$. 

Now then, the differential $\partial_2$ restricted to the summand $H_0(\Hom(\Z^3,SO(3))_\one)$ of $A_3$ is the alternating sum of 4 identity homomorphisms which is then the zero homomorphism. Now we analyze the values of the face maps on the the connected components that do not contain $\one$. 
Each such component has the form $\{g(x_1, x_2, x_3) g^{-1}: g\in SO(3)\}$ for some $3$--tuple $(x_1, x_2, x_3)\in SO(3)^3$ with $\langle x_1, x_2, x_3\rangle \isom D_4 = \Z/2\oplus\Z/2$
(this follows from the arguments in \cite[Section 3]{AnCVi}).
To write canonical representatives of these components, choose elements distinct, non-identity elements $c_1,c_2, c_3 \in SO(3)$ such that $\langle c_1, c_2, c_3\rangle \isom D_4$ (so $c_3 = c_1 c_2$). The action of the quotient group $N_{SO(3)}(D_4)/D_4\cong\Sigma_3$ by conjugation on $D_4$ is by permuting the elements $c_j$.  The 7 remaining components in $C_3(SO(3))$ (all homeomorphic to $SU(2)/Q_8$) are then represented by the 3-tuples
\[(c_1,c_2,I),(c_1,c_2,c_2),(I,c_2,c_3),(c_2,c_2,c_3),(c_1,I,c_3),(c_1,c_2,c_1),(c_1,c_2,c_3)\] 
To compute the values of $(d_i)_*$ on these components, we only need to know the connected component of $C_2(SO(3))$ where the image of $d_i$ lands. For our purposes let us write the obvious identities $c_1c_2=c_3$, $c_1c_3=c_2$ and $c_2c_3=c_1$ in $D_4$. Using that the pairs of the form $(I,c_j)$, $(c_j,I)$ and $(c_j,c_j)$ are in $\Hom(\Z^2,SO(3))_\one$ for $1\leq j\leq 3$, it is easy to check that the alternating sum of $(d_i)_*$, with $0\leq i\leq 3$ at each of the 5 components $H_0(SU(2)/Q_8)$ represented by $(c_1,c_2,I),(I,c_2,c_3),(c_1,I,c_3),(c_1,c_2,c_1)$ and $(c_1,c_2,c_3)$ is zero. For the ones represented by $(c_1,c_2,c_2)$ and $(c_2,c_2,c_3)$, the image of the alternating sum of $(d_i)_*$ is the ideal generated by $(-2,2)$. We include both computations in Tables \ref{tab:table1} and \ref{tab:table2}, where it can be seen that the alternating sum of $(d_i)_*$ is as claimed. We conclude that $E_{2,0}^2=\langle(-1,1)\rangle/\langle(-2,2)\rangle=\Z/2$.
\end{proof}

\begin{table}
\begin{center}
\caption{$\partial$ restricted to $H_0(SU(2)/Q_8)$ with zero image}
    \label{tab:table1}
\begin{tabular}{|c|c|c|c|c|c|}
 \hline
             &$(c_1,c_2,I)$ & $(c_1,c_2,c_1)$ & $(I,c_2,c_3)$ & $(c_1,I,c_3)$  & $(c_1,c_2,c_3)$\\ 
 \hline
 $(d_0)_*$   &$(1,0)$ &$(0,1)$ &$(0,1)$ &$(1,0)$ &$(0,1)$ \\  
 \hline
 $(d_1)_*$   &$(1,0)$ &$(0,1)$ &$(0,1)$ &$(0,1)$ &$(1,0)$ \\   
 \hline
 $(d_2)_*$   &$(0,1)$ &$(0,1)$ &$(1,0)$ &$(0,1)$ &$(1,0)$ \\ 
 \hline
 $(d_3)_*$   &$(0,1)$ &$(0,1)$ &$(1,0)$ &$(1,0)$ &$(0,1)$ \\  
 \hline
\end{tabular}
\end{center}
\end{table}

\begin{table}
\begin{center}
\caption{$\partial$ restricted to $H_0(SU(2)/Q_8)$ with non-zero image}
    \label{tab:table2}
\begin{tabular}{|c|c|c|}
 \hline
             &$(c_1,c_2,c_2)$ & $(c_2,c_2,c_3)$\\ 
 \hline
 $(d_0)_*$   &$(1,0)$ & $(0,1)$  \\  
 \hline
 $(d_1)_*$   &$(0,1)$ & $(1,0)$  \\   
 \hline
 $(d_2)_*$   &$(1,0)$ & $(0,1)$  \\ 
 \hline
 $(d_3)_*$   &$(0,1)$ & $(1,0)$  \\  
 \hline
\end{tabular}
\end{center}
\end{table}

\begin{proposition}\label{prop:pi2BcomOn}
\
\begin{enumerate}

\item $\pi_2(\Bcom SO(3))=\pi_2(\Bcom O(3))=\Z/2\oplus\Z/2$;

\item For $n\geq2$, the inclusions $\Bcom O(n)\to\Bcom O(n+1)$ are 2-connected and for $n\geq3$ they induce isomorphisms in $\pi_2$, and

\item For any $n\geq 3$, the inclusions $\Bcom SO(n)\to\Bcom SO(n+1)$ induce isomorphisms in $\pi_2$.

\end{enumerate}
\end{proposition}

\begin{proof}
(1) Since $SO(3)$ is path-connected, $\Bcom SO(3)$ is simply connected. By the Hurewicz Isomorphism Theorem and Lemma \ref{lem:H2Bcomso3}, $\pi_2(\Bcom SO(3))\cong\Z/2\oplus\Z/2$. Recall that $O(3)\cong \{\pm I\}\times SO(3)$, so that $\Bcom O(3)\cong B\{\pm I\}\times\Bcom SO(3)$. Since $B\{\pm I\}$ is an Eilenberg-MacLane space of type $K(\Z/2,1)$, we see that
\begin{align}
\pi_2(\Bcom O(3))\cong\Z/2\oplus\Z/2.\label{eq:p2o3so3}
\end{align}

(2)-(3) We claim that for any $n\geq3$ the inclusions $\Bcom O(n)\to \Bcom O(n+1)$ are 2-connected (note, in particular,  that this implies  surjectivity at the level of $\pi_2$). We have already mentioned that $C_\bullet(G)$ is a proper simplicial space for any Lie group $G$, and by \cite[Appendix A]{Segal}, the \emph{fat} geometric realization $\|C_\bullet(G)\|$ is equivalent to $\Bcom G$. Then by \cite[Lemma 2.4]{EbRW} to prove our claim for the orthogonal groups, we only need to show that $C_k(O(n))\to C_k(O(n+1))$ is $2-k$ connected for $k=0,1$ and 2. The case $k=0$ is trivial since $C_0(G)=\{pt\}$. For $k=1$, it is well known that the inclusions $O(n)\to O(n+1)$ induce isomorphisms in fundamental groups and bijections at the level of $\pi_0$. By \cite[Theorem 1.1]{HRojo} we have a bijection $\pi_0(C_2(O(n)))\to\pi_0(C_2(O(n+1)))$ whenever $n\geq 2^2-1=3$. Therefore $\pi_2(\Bcom O(n))\to\pi_2(\Bcom O(n+1))$ is surjective. 

Similarly, for every $n\geq3$, using the well known behavior of the inclusions $SO(n)\to SO(n+1)$ in $\pi_1$ and \cite[Corollary 1.5]{HRojo} we see that 
$$\Bcom SO(n)\maps \Bcom SO(n+1)$$ is also 2-connected. Therefore 
$\pi_2(\Bcom SO(n))\to\pi_2(\Bcom SO(n+1))$ is surjective as well.

To prove part $(2)$, first notice that part (1) and Proposition \ref{thm:scEcomO} for $n=3$ imply that $\pi_2(\Bcom O(2))\to \pi_2(\Bcom O(3))$ is surjective. The isomorphism (\ref{eq:p2o3so3}) and surjectivity at $\pi_2$ of the inclusions $\Bcom O(n)\to\Bcom O(n+1)$ imply that the groups $\pi_2(\Bcom O(n))$ are all quotients of $\Z/2\oplus\Z/2$. By Proposition \ref{thm:scEcomO}, for any $n\geq 3$ we see two distinct non-trivial group elements in $\pi_2(\Bcom O(n))$ (e.g. $[j_n\circ f_1]$ and $[j_n\circ(j\circ g_1)]$ as in the notation of Theorem \ref{thm:csEk} and Proposition \ref{thm:scEcomO}). Thus the only possible quotient is $\Z/2\oplus\Z/2$, and then the isomorphisms in part (2) follow.  

Naturality of the isomorphisms $\{\pm I\}\times SO(2n+1)\cong O(2n+1)$ with respect to the aforementioned inclusions and the induced isomorphisms at the level of $\pi_2$ also give the claimed isomorphisms in part (3). 
\end{proof}

\begin{corollary}\label{cor:H2bcomon}
For every $n\geq 2$, the inclusions $\Bcom O(n)\to\Bcom O(n+1)$ induce isomorphisms $H^2(\Bcom O(n+1);\F_2)\xrightarrow{\cong}H^2(\Bcom O(n);\F_2)$.
\end{corollary}
\begin{proof}
Part (2) of Proposition \ref{prop:pi2BcomOn} implies that for $n\ge2$, the homotopy fiber of the inclusions $\Bcom O(n)\to\Bcom O(n+1)$ is simply connected. Then the associated  five--term Serre exact sequence for cohomology (see \cite[Corollary 9.14]{DavisKirk} for the homology version) implies that 
$$H^2(\Bcom O(n+1);\F_2)\xrightarrow{}H^2(\Bcom O(n);\F_2)$$ 
is injective. It remains to show that all these groups have the same rank. Since $\Bcom SO(n)$ is simply connected, to compute its $\F_2$-cohomology for $n\geq 3$ we can use parts (1) and (3) of Proposition \ref{prop:pi2BcomOn}, and together with the K\"unneth formula we see that
\begin{align*}
H^2(\Bcom O(2n+1);\F_2)
 & \cong H^2(\Bcom SO(2n+1);\F_2)\oplus H^2(\RP^\infty;\F_2)\\
 &\cong (\F_2\oplus \F_2)\oplus \F_2.
\end{align*}
Hence for $n\geq 3$ and odd, the groups $H^2(\Bcom O(n);\F_2)$ have rank 3. Injectivity for all $n\geq 3$ implies now that $H^2(\Bcom O(n+1);\F_2)\xrightarrow{}H^2(\Bcom O(n);\F_2)$ are isomorphisms. By \cite[Theorem 5.1]{AnCGrVi}, $H^2(\Bcom O(2);\F_2)$ has rank 3 and thus $H^2(\Bcom O(3);\F_2)\xrightarrow{}H^2(\Bcom O(2);\F_2)$ is an isomorphism as well. 
\end{proof}

Corollary~\ref{cor:H2bcomon} allows us to define a new characteristic class $a_2$ for commutative bundles. We will use this class in several subsequent arguments.

\begin{definition}\label{def: a_2}
Let $\bar{r} \in H^2(\Bcom O(2);\F_2)$ denote the reduction mod 2 of $r\in H^2(\Bcom O(2);\Z)$ (as in (\ref{eq:r&e})).
By Corollary \ref{cor:H2bcomon}, there exists a unique class $a_2\in H^2(\Bcom O;\F_2)$ satisfying $i^*(a_2) = \bar{r}$, where $i\co \Bcom O(2)\injects \Bcom O$ is the inclusion. (The restriction of $a_2$ to a class on $\Bcom O(n)$ will also be denoted by $a_2$.)
\end{definition}

\subsection{Additive structure of $\rKOcom(\Sigma)$}\label{sec:add}

Gritschacher, using very different methods from those in the present paper, established an isomorphism 
\[\widetilde{KO}_\text{com}(S^2) \cong \Z/2 \oplus\widetilde{KO}(S^2) = \Z/2 \oplus\Z/2.\]
Using Proposition~\ref{prop:pi2BcomOn}, we now give a new proof of this isomorphism, and identify generators for this group arising from our transitionally commutative $O(2)$--bundles.

Let $E_0^{f_1}$ denote the trivial real vector bundle of rank 2 over $S^2$ with transitionally commutative structure $f_1$. Let $E_1$ be the oriented vector bundle classified by the generator $1\in \pi_2(BO(2))\isom \Z$, and let $E_1^{\st}$ denote $E_1$ with its algebraic TC structure $S^2\xmaps{g_1}BSO(2)\xmaps{j}\Bcom O(2)$. 

\begin{proposition} \label{prop:mainS2}
The commutative $K$--theory classes associated to $E_0^{f_1}$ and $E_1^{\st}$ generate the group $$\widetilde{KO}_\mathrm{com}(S^2)\cong \Z/2 \oplus \Z/2.$$ 
\end{proposition}
\begin{proof} By Proposition~\ref{prop:pi2BcomOn}, $\pi_2(\Bcom O)\cong \Z/2\oplus\Z/2$. Since $\Bcom O$ is a (connected) H--space under block sum of commuting matrices, we have an isomorphism 
$$\widetilde{KO}_\mathrm{com}(S^2) = [S^2, \Bcom O] \isom \pi_2 (\Bcom O).$$
Next, note that any two distinct non-zero elements of  $\Z/2 \oplus \Z/2$ generate this group, so it suffices to show that the classes associated to $E_0^{f_1}$ and $E_1^{\st}$ are non-zero and distinct. The class associated to $E_0^{f_1}$ is non-zero by Proposition~\ref{thm:scEcomO}. The class associated to
$E_1^{\st}$ is non-zero since 
$$S^2\srm{g_1} BSO(2) \injects BO(2)$$
 is a generator of $\pi_2 (BO(2))$, and the map $\pi_2 (BO(2))\to \pi_2 (BO)$ is surjective. Since the composition $S^2 \srm{f_1} \Bcom O(2) \injects BO(2)$ is null-homotopic, these classes are distinct.
\end{proof}

 The inclusion $\iota\co \Bcom O\to BO$ induces, for each finite CW complex $X$, a map $\iota_* \co \rKOcom (X) \to \widetilde{KO} (X)$. The kernel of this map is the ``non-standard" part of the commutative $K$--theory of $X$. In~\cite{Ad3}, it is shown that there is a splitting of infinite loop spaces
$$\Bcom O \simeq BO \times \Ecom O.$$
In fact, it follows from~\cite{Ad3} that the natural map 
\begin{align} \label{eq:ns}[X, \Ecom O]\to \ker(\iota_*)\end{align}
is an isomorphism.

 Let $i\colon \Bcom O(2)\to \Bcom O$ be the inclusion, and for a closed connected surface $\Sigma$, let $c\colon \Sigma\to S^2$ denote the map that collapses the 1-skeleton of the standard cell decomposition of $\Sigma$ (with a single 2-cell).

\begin{theorem}\label{thm: main}  
Let $\Sigma$ be a closed connected surface. Then \[\rKOcom(\Sigma)\cong \widetilde{KO}(\Sigma)\oplus \Z/2,\] 
and the kernel of $\iota_* \co \rKOcom(\Sigma)\to \widetilde{KO}(\Sigma)$ is generated 
by $i\circ f_1\circ c\colon \Sigma\to \Bcom O$, and $(i\circ f_1\circ c)^*(a_2)$ is the generator in $H^2(\Sigma;\F_2)$.
\end{theorem}
\begin{proof}  By (\ref{eq:ns}) we have a natural isomorphism $\ker (\iota_*) \isom [\Sigma, \Ecom O]$.
To avoid confusion, for the remainder of the proof let us write $\iota_*^{\Sigma} \co \rKOcom(\Sigma)\to \widetilde{KO}(\Sigma)$ instead of $\iota_*$. Consider the cofibration sequence $\bigvee_{m}S^1\to \Sigma\xrightarrow{c}S^2$, which induces the sequence
\[[S^2,\Ecom O]\xrightarrow{- \circ c}[\Sigma,\Ecom O]\to \left[\bigvee_{m}S^1,\Ecom O\right].\]
Since the map $\pi_2(\Bcom O)\to\pi_2(BO)$ induced by $\iota$ is surjective and $\pi_1(\Bcom O)\cong \pi_1(BO)$ (see \cite[Lemma 4.3]{Ad3}), $\Ecom O$ is simply connected, and then the last term in the above sequence is trivial. Therefore $[S^2,\Ecom O]\xrightarrow{- \circ c}[\Sigma,\Ecom O]$ is surjective.

Taking $\Sigma = S^2$, the proof of Proposition~\ref{prop:mainS2} shows that
$$\iota^{S^2}_*\co \rKOcom(S^2) \to \rKO (S^2)$$
is a surjection $\Z/2\oplus \Z/2\twoheadrightarrow \Z/2$, so $\ker (\iota^{S^2}_*) \isom \Z/2$.
Moreover, as above we have $\ker (\iota^{S^2}_*) \isom [S^2, \Ecom O]$.
It follows that $[\Sigma,\Ecom O]$ has at most two elements, and to complete the proof, it suffices to show that $[i\circ f_1\circ c]$ is a non-trivial element of $\ker (\iota_*^{\Sigma}) \isom [\Sigma,\Ecom O]$.

 By Lemma~\ref{lem:to2b}, we have $[i\circ f_1\circ c]\in \ker (\iota_*^{\Sigma})$. To show that $[i\circ f_1\circ c]$ is non-trivial, we study the values of the composition 
\[\Sigma\xrightarrow{c}S^2\xrightarrow{f_1}\Bcom O(2)\xrightarrow{i}\Bcom O\] 
in $\F_2$-cohomology. Recall from Definition~\ref{def: a_2} that we have a class 
$$a_2\in H^2(\Bcom O;\F_2)$$
that pulls back to $\bar{r}$ (the reduction mod 2 of $r\in H^2(\Bcom O(2);\Z)$). 
By Lemma \ref{lemm:clrEul}, $f_1^*(\bar{r})\in H^2(S^2;\F_2)$ is the generator. Finally, since $c^*\colon H^2(S^2;\F_2)\to H^2(\Sigma;\F_2)$ is an isomorphism,  
$(i\circ f_1\circ c)^*(a_2)$ is the generator in $H^2(\Sigma;\F_2)$. 
\end{proof}

\begin{remark} Theorem~\ref{thm: main} can be proven without  (\ref{eq:ns}). The above argument shows  $|\ker (\iota_*)|\geq 2$  and  $|[\Sigma, \Ecom O]|\leq 2$, and the homotopy lifting property for the homotopy fibration $\Ecom O \to \Bcom O \srt{\iota} BO$ yields a surjection $[\Sigma, \Ecom O]\twoheadrightarrow \ker (\iota_*)$.
\end{remark}

\section{Real commutative $K$-theory of surfaces: Multiplicative structure }

In this section we compute the ring structure of $\rKOcom (\Sigma)$ for closed, connected surfaces $\Sigma$. This will be achieved using the characteristic class $a_2$ introduced in Definition~\ref{def: a_2}. In order to work with this class, we will need some results regarding the effect of 
$$\phi^{-1}\colon\Bcom O(2)\to \Bcom O(2)$$
 in cohomology with $\F_2$--coefficients.

\subsection{The inverse map in $\F_2$ cohomology}

In \cite[Theorem 5.1]{AnCGrVi}, a presentation
\[H^*(\Bcom O(2);\F_2)\cong \F_2[w_1,w_2,\bar{r},s]/(w_1\bar{r}, \bar{r}^2,\bar{r}s, s^2)\]
is given, where $w_1$ and $w_2$ are the pullbacks of the first and second Stiefel--Whitney classes  along the inclusion $\iota\co \Bcom O(2)\to BO(2)$;  the class $\bar{r}$ is the reduction of the class $r\in H^2(\Bcom O(2); \Z)$ described in (\ref{eq:r&e}); and  $s\in H^3(\Bcom O(2);\F_2)$.
 It is also shown there that the action of the Steenrod algebra is determined by its action on $H^*(BO(2);\F_2)$ and the total Steenrod squares $\Sq(\bar r) = \bar r$ and 
\begin{align}
\Sq(s) = s + w_2\bar{r} + w_1^2s\label{eq:ssq}
\end{align}

\begin{proposition}\label{prop:invm2}
$(\phi^{-1})^*\co H^*(\Bcom O(2);\F_2)\to H^*(\Bcom O(2);\F_2)$ is given on generators by $w_1\mapsto w_1$, $w_2\mapsto w_2+\bar{r}$, $\bar{r}\mapsto\bar{r}$ and $s\mapsto s$.
\end{proposition}
\begin{proof}
We study the case of each class one at a time. 

$\bullet$ $w_1\mapsto w_1$. Since $(\phi^{-1})^*$ is an involution, it follows that $(\phi^{-1})^*(w_1)$ is not zero and hence $(\phi^{-1})^*(w_1) = w_1$. 

$\bullet$ $\bar{r}\mapsto \bar{r}$. To see this, notice that $(\phi^{-1})^*(r)=-r$ in $H^2(\Bcom O(2);\Z)$ (which is shown similarly as in the proof of Lemma \ref{lemm:clrEul}). Therefore $(\phi^{-1})^*(\bar{r}) = \bar{r}$. 

$\bullet $ $w_2\mapsto w_2+\bar{r}$. We write $(\phi^{-1})^*(w_2)$ as a generic element  $(\phi^{-1})^*(w_2) = \epsilon_1w_1^2 + \epsilon_2w_2 + \epsilon_3\bar{r}$ with $\epsilon_i\in \{0,1\}$. To show that $\epsilon_3$ is non-zero, recall that Proposition \ref{prop:cco2b} for $k=-1$ and $n=-1$ says that $\iota\circ\phi^{-1}\circ f_{-1}$ has degree 1, and in particular, $f_{-1}^*(\phi^{-1})^*\iota^*(w_2)$ is the generator $\tau \in H^2(S^2;\F_2)$. Therefore
\[\tau=f_{-1}^*(\phi^{-1})^*\iota^*(w_2) = f_{-1}^*(\phi^{-1})^*(w_2)=f_{-1}^*(\epsilon_1w_1^2 + \epsilon_2w_2 + \epsilon_3\bar{r})\] Now then, Proposition \ref{prop:cco2b} for $n=1$ implies that $\iota\circ f_k$ is homotopy trivial, and for $i=1,2$, we see that $f_{-1}^*(w_i) = 0$. Therefore $(\phi^{-1})^*(w_2)$ must have the class $\bar{r}$ and thus $\epsilon_3=1$. For the value of $\epsilon_2$, once more, using that $\phi^{-1}$ is an involution we see the equation
\begin{align*}
w_2 &= (\phi^{-1})^* ((\phi^{-1})^* (w_2))\\
    &= (\phi^{-1})^*(\epsilon_1w_1^2 + \epsilon_2w_2 + \bar{r})\\
    &=\epsilon_1w_1^2 + (\epsilon_2\epsilon_1w_1^2 + \epsilon_2\epsilon_2w_2 + \epsilon_2\bar{r}) + \bar{r}
\end{align*}
which implies $\epsilon_2=1$. Lastly, to show that $\epsilon_1=0$, we use the inversion map restricted to the abelian subgroup $O(1)^2\subset O(2)$. Any element in $O(1)\times O(1)$ has order 2, so that $(-)^{-1}\co O(1)^2\to O(1)^2$ is the identity and hence $\phi^{-1}\co BO(1)^2\to BO(1)^2$ is the identity as well. Since $O(1)^2$ is abelian, the inclusion $i\co BO(1)^2\hookrightarrow BO(2)$ factors through $\Bcom O(2)$, and we see the commutative diagram
\[
\xymatrix{
BO(1)^2\ar[d]^k\ar@{=}[r]^{\phi^{-1}}& BO(1)^2\ar[d]^k\ar@/^/[rd]^i &\\
\Bcom O(2)\ar[r]^{\phi^{-1}} & \Bcom O(2)\ar[r]^{\iota} & BO(2).
}
\]
Fix a presentation $H^*(BO(1)^2;\F_2)=\F_2[u,v]$ such that $i^*$ gives the isomorphism with the invariants subring $\F_2[u,v]^{\Z/2} \cong H^*(BO(2);\F_2)$, where the invariants are taken under the action of permuting $u$ and $v$, that is, $i^*(w_1)=u+v$ and $i^*(w_2)=uv$.  Applying $H^*(-;\F_2)$ to the diagram, we see that $(\phi^{-1})^*i^*(w_2)=uv$ must equal $k^*(\phi^{-1})^*\iota^*(w_2) = k^*(\epsilon_1w_1^2 + w_2 + \bar{r})=\epsilon_1(u^2+v^2) + uv + k^*(\bar{r})$, where the last equality holds by commutativity of the triangle in the diagram. It is shown in \cite[~p 23]{AnCGrVi} that $k^*(r) = 0$, so we can conclude $\epsilon_1 = 0$.

$\bullet$ $s\mapsto s$. Following the same procedure as above, let us write $(\phi^{-1})^*(s)=\epsilon_1 w_1^3 + \epsilon_2 w_1w_2 + \epsilon_3 s$ with $\epsilon_i\in \{0,1\}$. Similarly, $\phi^{-1}$ being an involution gives $\epsilon_3=1$. To analyze $\epsilon_1$ and $\epsilon_2$, we apply $k^*$ to both sides of the equation, and we use the fact that $k^*(s) = 0$ (which is shown in  \cite[~p 25]{AnCGrVi}) and $k^* \circ (\phi^{-1})^* = k^*$ to obtain $0 = \varepsilon_1(u+v)^3 + \varepsilon_2(u+v)uv$ in $H^*(BO(1)^2; \F_2)$. It follows that $\varepsilon_1 = \varepsilon_2 = 0$.
\end{proof}

\begin{remark}\label{eq:stablerbar}
Recall that the inclusion $BO(1)^n\to BO(n)$ induces an injective map in cohomology with $\F_2$-coefficients (by the Splitting Principle), and since $O(1)^n$ is abelian, the inclusion factors through $\iota\colon \Bcom O(n)\to BO(n)$. This yields a commutative diagram
\[\xymatrix{H^*(BO(n);\F_2)\ar[r]^{\iota^*}\ar@/_/[rd] & H^*(\Bcom O(n);\F_2)\ar[d]\\
&H^*(BO(1)^n;\F_2)},\]
 and thus $\iota^*$ is injective. Then for any $1\leq i\leq n$ we can define Stiefel--Whitney classes in $H^*(\Bcom O(n);\F_2)$ as the pullback under $\iota$ of the ordinary Stiefel--Whitney classes $w_i\in H^i(BO(n);\F_2)$; that is
\[w_i:=\iota^*(w_i)\in H^i(\Bcom O(n);\F_2).\]
Definition~\ref{def: a_2} provides a class $a_2\in H^2(\Bcom O(n);\F_2)$ such that $j_n^*(a_2)=\bar r\in H^2(\Bcom O(2);\F_2)$. 
Consider $w_2\in H^2(\Bcom O(n);\F_2)$. Then \ref{prop:invm2} implies that \[(\phi^{-1})^*(w_2)=w_2+a_2,\]
and the analogous formula holds in $H^2(\Bcom O;\F_2)$ as well.
\end{remark}

\begin{remark}\label{rmk: a_2}(\emph{Obstruction to algebraic TC structures}) 
Given a TC structure $f\colon X\to \Bcom O(n)$, it is interesting to ask whether $f$ is algebraic. Recall that this amounts to asking whether $f$ factors (up to homotopy) through the inclusion $i_A \co BA\hookrightarrow \Bcom O(n)$ associated to an abelian subgroup $A\leq O(n)$.

We claim that the class $a_2$ is an obstruction to algebraicity. It can be shown that every abelian subgroup $A\subset O(n)$ is conjugate to a subgroup of the product $O(1)^{n-2k}\times SO(2)^k$ (thought of as subgroup of block--diagonal matrices in $O(n)$). Hence we may assume, without loss of generality, that $A=O(1)^{n-2k}\times SO(2)^k$. Consider a map $g\colon X\to BA$. Then 
$$(i_A \circ g)^*(a_2)=(i_A \circ g)^*(w_2)+(\phi^{-1}\circ i_A \circ g)^*(w_2).$$ 
Now $\phi^{-1}\circ i_A=i_A\circ\phi^{-1}$, and 
 $$(\phi^{-1})^{*}\colon H^2(BA;\F_2)\to H^2(BA;\F_2)$$ 
 is the identity, since $\phi^{-1}\co O(1)\to O(1)$ is the identity and    Lemma \ref{lem:powerophomi} implies that $(\phi^{-1})^{*}\colon H^2(BSO(2);\F_2)\to H^2(BSO(2);\F_2)$ is the identity.
It now follows that $(\phi^{-1}\circ i_A)^*=i_A^*$ and then $(i_A\circ g)^*(a_2)=0$. In other words, if $E$ is a bundle with an algebraic TC structure $f$, then $a_2(E^{f})=0$.
\end{remark}

\subsection{Ring structure of $\rKOcom(\Sigma)$} Block sum of matrices and Kronecker product define homomorphisms $\oplus\colon O(n)\times O(m)\to O(n+m)$ and $\otimes\colon O(n)\times O(m)\to O(mn)$, which induce maps $\Bcom(\oplus)$ and $\Bcom(\otimes)$ making $\left[X,\coprod_{n=0}^\infty \Bcom O(n)\right]$  a \emph{semi-ring}, and for $X$ a compact space, we may consider the Grothendieck ring 
\[\KOcom(X):=\text{Gr}\left[X,\coprod_{n=0}^\infty \Bcom O(n)\right].\]
For $n\geq 0$ the inclusions $\Bcom O(n)\to \{n\}\times\Bcom O$ induce an isomorphism of rings $\KOcom(X)\xmaps{\cong}[X,\Z\times\Bcom O]$ (see \cite[Theorem 5.5]{Ad3}). Recall that our definition of reduced commutative $K$-theory is $\rKOcom(X)=[X,\Bcom O]$, which now under this isomorphism, can be identified with the non-unital ring generated by the formal differences $[f]-[*_n]$, where $*_n,f\colon X\to \Bcom O(n)$ and $*_n$ is the null-map. For our convenience we will represent these classes as formal differences of vector bundles with a fixed TC structure. Let $(\varepsilon^n)^{\st}$ denote the trivial bundle $\varepsilon^n$ with TC structure given by $*_n$, and let $E^f$ denote a rank $n$ vector bundle $E$ with TC structure $f$ (the underlying bundle $E$ is classified by the map $\iota\circ f\colon X\to BO(n)$). Then the stable class $E^f$ (or $[f]$) in $\rKOcom(X)$ is represented by $E^f-(\varepsilon^n)^{\st}$.

Recall that by Theorem~\ref{thm: main},
there is a unique non-zero element in the kernel of the natural map 
$\rKOcom(\Sigma)\to \rKO(\Sigma)$.  We refer to this element as \e{the non-standard stable class}.
 As an application of Proposition \ref{prop:invm2}, we show that for a closed connected surface $\Sigma$ all products in $\rKOcom(\Sigma)$ with the non-standard stable class are trivial.  
Our calculations will use the next two formulas, which may be verified using  the splitting principle. Let  $E$ and $F$ be rank $2$ real vector bundles, and let  $L$ be a real line bundle, all over the same base space. Then
\begin{align}
w_2(E\otimes F)&=w_1(E)^2+w_1(E)w_1(F)+w_1(F)^2\label{eq:w2}\\ 
w_2(E\otimes L)&=w_1(E)^2+w_1(L)^2+w_2(E).\label{eq:w2L}
\end{align}

\begin{theorem}\label{thm:main2}
For every closed connected surface $\Sigma$, there is an isomorphism of non-unital rings
\[\rKOcom(\Sigma)\cong \rKO(\Sigma)\cross \langle y\rangle\] 
where $2y=y^2=0$.
\end{theorem}
\begin{proof}  
In the Appendix, we show that for non-simply connected surfaces, $\rKO(\Sigma)$ is generated (as a ring) by classes $[L]-[\epsilon^1]$, where $L$ is a line bundle. 
The classifying map $\Sigma\to BO(1)$ of a line bundle $L$ is itself an algebraic TC structure on $L$ (since $O(1)$ is abelian), and we refer to this as the \e{standard} TC structure. Whitney sum and tensor product of line bundles have preferred algebraic TC structures as well, since in this case block sum of matrices and Kronecker product are homomorphisms of abelian groups; we again refer to these preferred structures as \e{standard}. 
Associating to each generator $l_1, \ldots, l_k$ of the presentation of $\rKO(\Sigma)$ given in Appendix B its standard (algebraic) TC structure yields a commutative diagram of ring homomorphisms

\[\xymatrix{\Z[l_1, \ldots, l_k]\ar@/_/[rd]^\pi\ar[r]^s &\KOcom(\Sigma)\ar[d]^{\iota_*}\\
&KO(\Sigma)}\]

By commutativity, the kernel of $\pi$ maps to the kernel of $\iota_*$ (note here that this kernel is the same whether we work with reduced or unreduced $K$--theory, and hence we use the notation $\iota_*$ in both settings). Each class in the image of $s$ is algebraic, and by Remark~\ref{rmk: a_2} we know that $a_2$ vanishes on all such classes. By Theorem~\ref{thm: main} there is only one non-zero class in $\ker(\iota_*)$, and $a_2$ is non-zero on this class. It follows that $s$ must 
be zero on the kernel of $\pi$, and hence $s$ factors through $\Z[l_1, \ldots, l_k]/\ker (\pi) \isom KO(\Sigma)$. We thus have induced maps 
$$\sigma \co KO(\Sigma) \maps \KOcom(\Sigma)$$
and
$$\widetilde{\sigma} \co \rKO(\Sigma) \maps  \rKOcom(\Sigma)$$
such that the composition
\begin{align*}
KO(\Sigma)\xrightarrow{\sigma} \KOcom(\Sigma)\xrightarrow{\iota_*} KO(\Sigma)
\end{align*}
is the identity, and similarly for $\widetilde{\sigma}$. This implies that we have a splitting of abelian groups
$$\rKOcom (\Sigma) \isom \rKO (\Sigma) \oplus \ker (\iota_*),$$
and also shows that summand of $\rKOcom (\Sigma)$ corresponding to $\rKO (\Sigma)$  (that is, the image of $\widetilde{\sigma}$) is a subring of $\rKOcom (\Sigma)$. Since $ \ker (\iota_*)$ is isomorphic to $\Z/2$ as an abelian group,  to complete the proof it suffices to show that for every element $x\in \rKOcom (\Sigma)$, the product of $x$ with the generator $y\in \ker (\iota_*)$ is zero.

The case $\Sigma=S^2$ follows from the general fact that if $\widetilde{E}^*$ is a reduced multiplicative cohomology theory and $Y$ is a suspension, then all products in the ring $\widetilde{E}^*(Y)$ are zero.
For the non-simply connected case, by Theorem \ref{thm: main} it is enough to show that 
\begin{equation}\label{eqn:2} 
(c^*(E_0^{f_1})-(\varepsilon^2)^{\st})^2=0
\end{equation}
and that for an arbitrary line bundle $L$, 
\begin{equation}\label{eqn:prod}
(c^*(E_0^{f_1})-(\varepsilon^2)^{\st})(L^{\st}-(\varepsilon^1)^{\st})=0.
\end{equation}
Equation (\ref{eqn:2}) follows at once, since the collapse map $\Sigma\xmaps{c} S^2$ induces a ring homomorphism $c^*\colon \rKOcom(S^2)\to\rKOcom(\Sigma)$. To prove (\ref{eqn:prod}), it suffices to show that 
$(c^*E_0^{f_1}\otimes L^{\st})$ and $(2L^{\st}\oplus c^*E_0^{f_1})$ are stably equivalent. Both underlying bundles are stably equivalent to $2L$, and since there is a single non-standard class in $\rKOcom(\Sigma)$ that is detected by $a_2$, we only have to calculate the values of $a_2$ on these bundles. To compute the values of $a_2$, recall from Remark \ref{eq:stablerbar}  that $$a_2(E^{f})=w_2(E)+w_2(\phi^{-1}(E^{f})),$$ where the TC structure on $\phi^{-1}(E^{f})$ is $\phi^{-1}\circ f$ and the underlying bundle is classified by $\iota\circ \phi^{-1}\circ f$. By Proposition \ref{prop:cco2b} the map $\iota\circ\phi^{-1}\circ f_1$ is homotopic to the map that classifies a bundle with choice of orientation $-1$. Hence the underlying bundle of $\phi^{-1}(E_0^{f_1})$ is $E_{-1}$. 
We now compute $a_2(c^*E_0^{f_1}\otimes L^{\st})$
using formula (\ref{eq:w2L}), along with the fact that $\phi^{-1}(E^f\otimes F^g)=\phi^{-1}(E^f)\otimes\phi^{-1}(F^g)$ for any pair of vector bundles $E$ and $F$ with respective TC structures $[f]$ and $[g]$, and the fact that $\phi^{-1}(L^{\st})=L^{\st}$: We have
\begin{align*}
a_2(c^*E_0^{f_1}\otimes L^{\st})
&=w_2(2L)+w_2(c^*E_{-1}\otimes L)\\\
&=w_2(2L)+c^*(w_2(E_{-1})) + w_1(L)^2\\
&=c^*(w_2(E_{-1})).
\end{align*}
One can readily verify that $a_2(2L^{\st}\oplus c^*E_0^{f_1})=c^*(w_2(E_{-1}))$ as well. 
\end{proof}

\appendix

\section{Real topological $K$-theory of surfaces}

We now describe the relationship between the $\F_2$-cohomology ring of a closed connected surface and its  real topological $K$-theory, yielding the   presentations
\[KO(S^2)\isom \Z[e_1]/(2e_1,e_1^2); 
\]
\[KO(\Sigma_g)\isom\Z[l_{a_i},l_{b_j}:  1\leq i, j\leq g]/(2l_{a_i},2l_{b_j},l_{a_i}l_{a_j},l_{b_i}l_{b_j},l
_{a_i}l_{b_i}+l_{a_j}l_{b_j},l_{a_i}l_{b_k}:i\ne k), 
\]
and
\[KO(P_n)\isom\Z[l_{a_i}:1\leq i \leq n]/(4l_{a_i},l_{a_i}^2-2l_{a_j},l_{a_i}l_{a_k}: i\ne k),
\]
where $\Sigma_g$ is the connected, orientable surface of genus $g>0$ and $P_n$ is the connected sum of $n$ copies of $\RP^2$.

 \subsection{Ring presentations}\label{sec:pres} Let $R$ be a unital, commutative ring of characteristic zero that is additively generated by elements $r_1, \ldots, r_n\in R$. Then there is 
a surjective ring homomorphism
 $f\co \Z[x_1, \ldots, x_n]\to R$ sending $x_i$ to $r_i$.  
We have
$r_i r_j =  \sum_k a_k^{ij} r_k$
for some $a_k^{ij} \in \Z$ ($i,j\in \{1, \ldots, n\}$), and hence 
\begin{align}x_ix_j - \sum_k a_k^{ij}  x_k\in \ker (f).
\label{rel1}
\end{align}
By eliminating generators if necessary, we may assume that the only relations
$\sum_k a_k r_k = 0$ ($a_1, \ldots a_n\in \Z$) are of the form $a_i r_i = 0$.
A simple induction on degree shows that $\ker (f)$ is  generated by the elements (\ref{rel1}) together with one
element $a_i x_i$ for each $i$ (where $a_i$ may be zero),  yielding a finite presentation of $R$.
If $R$ has characteristic $p>0$, there is a similar presentation with $\F_p$ in place of $\Z$.

More generally, say $r_1, \ldots, r_n$ generate $R$ as a ring (but not necessarily as an abelian group), and define $f\co \Z[x_1, \ldots, x_n]\to R$ as above.
If 
$$r_1, \ldots, r_n, p_1 (r_1, \ldots, r_n),\ldots, p_k (r_1, \ldots, r_n)$$
 form an additive generating set for $R$ (where the $p_i$ are integer polynomials), then we obtain another surjection
 $g\co \Z[x_1, \ldots, x_n, p_1, \ldots, p_k]\to R$ (sending $x_i$ to $r_i$ and $p_j$ to $p_j (r_1, \ldots, r_n)$),
and $\ker (f)$ is
the image of $\ker (g)$
under the map 
$$\Z[x_1, \ldots, x_n, p_1, \ldots, p_k] \to \Z[x_1, \ldots, x_n]$$ 
fixing $x_i$ and sending $p_i$ to $p_i (x_1, \ldots, x_n)$.  
By eliminating redundant generators, we may assume that there are no linear relations involving more than one of the generators $r_i$, $p_j (r_1, \ldots, r_n)$.
The above procedure then gives a finite generating set for $\ker(g)$, and hence for $\ker (f)$.
 The presentations below are obtained this way.

\subsection{The total Stiefel--Whitney class and $\rKO(\Sigma)$}
Let $\Sigma$ be a closed connected surface.
Classes in the ungraded cohomology ring $H^*(\Sigma;\F_2)$ can be uniquely written  as $x_0+x_1+x_2$, where $x_i\in H^i(\Sigma;\F_2)$ for $i=0,1,2$. The multiplicative units  $H^*(\Sigma;\F_2)^\times$ form an abelian group under the cup product, and each unit has the form $1+x_1+x_2$.
The total Stiefel--Whitney class $W$ of real vector bundles over $\Sigma$ takes values in $H^*(\Sigma;\F_2)^\times$, and extends to a well--defined map 
$$W\colon KO(\Sigma)\to H^*(\Sigma;\F_2)^\times$$ 
given by $W(E-F)=W(E)W(F)^{-1}$. Moreover, since $H^*(\Sigma;\F_2)$ is a commutative ring, $W$ is a homomorphism. Restricting $W$ to $\rKO(\Sigma)$ and writing classes in the form $E-\varepsilon^n\in \rKO(\Sigma)$ where  $\varepsilon^n$ is the trivial bundle of rank $n=\text{rank}(E)$, we see that $W(E-\varepsilon^n)=W(E)$.

\begin{lemma}\label{lem:W}  
The map $W$ is an isomorphism of abelian groups
\[W\co \rKO(\Sigma)\xrightarrow{\cong} H^*(\Sigma;\F_2)^\times.\]
\end{lemma}

\begin{proof} 
Consider the cofiber sequence $\bigvee_{n}S^1\to \Sigma\xrightarrow{c} S^2$. The induced exact sequence in $K$--theory has the form
\[\Z/2=\rKO(S^2)\to \rKO(\Sigma)\to \rKO\left(\bigvee_{n}S^1\right)=(\Z/2)^{n},\]
which implies that $\rKO(\Sigma)$ has at most $2^{n+1}$ elements. Since 
$H^*(\Sigma;\F_2)^\times$ has exactly $2^{n+1}$ elements, 
it suffices to prove $W$ is surjective.
For every $x_1\in H^1(\Sigma;\F_2)$, there is a line bundle $L$ such that   
$W(L)=1+x_1$.  The  bundle $c^*E_1$ classified by the map $\Sigma\xrightarrow{c}S^2\xrightarrow{g_1}BSO(2)$ satisfies $w_1(c^*E_1)=0$, and $w_2(c^*E_1)\in H^2(\Sigma;\F_2)\isom \F_2$ is the generator. 
One now finds that $W(c^*E_1\oplus L)=1+x_1+w_2(c^*E_1)$. 
\end{proof}

\subsection{The abelian group structure of $H^*(\Sigma;\F_2)^\times$}
There is a general procedure for computing the $\F_2$-cohomology ring of a connected sum of surfaces $M_1\# M_2$ by analyzing the surjection in cohomology $H^*(M_1\vee M_2;\F_2)\to H^*(M_1\# M_2;\F_2)$ via the associated Mayer--Vietoris sequence.
One obtains the presentation
\[H^*(\Sigma_g;\F_2)\isom \F_2[a_1,...,a_{g},b_1,...,b_{g}]/(a_ia_j,b_ib_j,a_ib_i+a_jb_j,a_ib_k : i\ne k ),\]
where $\deg(a_i)=\deg(b_i)=1$.
A computation now shows that every element in $H^*(\Sigma_g;\F_2)^\times$ has order 2, which yields the following result.

\begin{lemma}\label{lem:cohg}
For every $g\geq 1$, there is an isomorphism of abelian groups \[H^*(\Sigma_g;\F_2)^\times\cong(\Z/2)^{2g+1}.\] Moreover, letting $y_2\in H^2 (P_n; \F_2)$ denote the generator, we have the following generating set for $H^*(\Sigma_g;\F_2)^\times$$:$
\[\{1+a_i, 1+b_j, 1+y_2 : 1\leq i,j\leq g\}.\]
\end{lemma}

Let $P_n$ denote the connected sum of $n$ copies of  $\RP^2$; note that $P_n$ is non-orientable, and in fact each closed, connected, non-orientable surface is homeomorphic to $P_n$ for some $n\geq 1$.  The graded $\F_2$--cohomology ring of $P_n$ has a presentation
\[H^*(P_n;\F_2)\isom \F_2[a_1,...,a_n]/(a_i^3, a_i^2+a_j^2,a_ia_k:i\ne k),\]
where $\deg(a_i)=1$. (For $n>1$, the relation $a_i^3 = 0$ follows from the other relations.)

\begin{lemma}
For every $n\geq 1$, there is an isomorphism of abelian groups \[H^*(P_n;\F_2)^\times\cong\Z/4\times(\Z/2)^{n-1}.\] Moreover, $\{1+a_i: 1\leq i\leq n\}$   generates $H^*(P_n;\F_2)^\times$.
\end{lemma}
\begin{proof}
A computation shows that the $2^{n-1}$ elements $1+a_{i_{1}}+\cdots+a_{i_k}$ where $k$ is \emph{odd} (and $i_1,\ldots, i_k$ are distinct) all have (multiplicative) order 4. Furthermore, if $y_2\in H^2 (P_n; \F_2)$ denotes the generator, then when $k$ is odd all elements of the form 
$$1+a_{i_{1}}+\cdots+a_{i_k}+y_2$$
also have order 4, 
 so we have $2(2^{n-1})=2^n$ elements of order 4 in $H^*(P_n;\F_2)^\times$. When $k>0$ is \emph{even}, the elements $1+a_{i_{1}}+\cdots+a_{i_k}$ and $1+a_{i_{1}}+\cdots+a_{i_k}+y_2$ (with $i_1,\ldots, i_k$ distinct) have order 2. This yields   
 $2^{n}-1$ elements of order 2 in $H^*(P_n;\F_2)^\times$.
 The only abelian group of order $2^{n+1}$ with 
 $2^{n}-1$ elements of order 2 and $2^n$ elements of order 4
  is $\Z/4\times(\Z/2)^{n-1}$. 
 \end{proof}

\subsection{The ring structure of $KO(\Sigma)$}

We study the products in $\rKO(\Sigma)$ using the total Stiefel--Whitney class. 
We view $\rKO(\Sigma)$ as the kernel of $KO(\Sigma)\to KO(\textrm{pt})$, so that elements in $\rKO(\Sigma)$ are represented by virtual bundles $E-\varepsilon^n$ with $n=\dim (E)$.

\paragraph{\bf Case $\Sigma=S^2$:} All products are trivial since $S^2$ is a suspension.

\paragraph{\bf Case $\Sigma=\Sigma_g$:}
Let $L_{a_i}-\varepsilon^1$, $L_{b_j}-\varepsilon^1$, and $E_{y_2}-\varepsilon^2$ in $\rKO(\Sigma_g)$ denote the inverse images under $W$ of the units $1+a_i$, $1+b_j$, and $1+y_2$, respectively (where, as above,  $y_2\in H^2 (\Sigma_g; \F_2)$ is the generator).  
By Lemmas~\ref{lem:W} and~\ref{lem:cohg}, 
 these classes additively generate $\widetilde{KO}(\Sigma_g)$, so the method from Section~\ref{sec:pres} yields a presentation of $\widetilde{KO}(\Sigma_g)$ with these elements as generators. After computing the relevant relations, the resulting presentation
 is readily reduced to that given above.
  
\paragraph{\bf Case $\Sigma=P_n$:} For $1\leq i\leq n$ let $L_{a_i}-\varepsilon^1\in\rKO(P_n)$ denote the inverse image under $W$ of the unit $1+a_i$. 
Again, these elements form an additive generating set for $\widetilde{KO}(P_n)$,
and after computing the relations the claimed presentation.

\end{document}